\documentclass[a4paper,12pt]{amsart}\usepackage{graphicx}\usepackage[top=25truemm,bottom=25truemm,left=30truemm,right=30truemm]{geometry}
\newcommand{\toukou}[1]{\ifx\TOUKOU\undefined\else{#1}\fi}%
\newcommand{\toukoudel}[1]{\ifx\TOUKOU\undefined{#1}\else\fi}%
\newcommand{\toukouchange}[2]{\ifx\TOUKOU\undefined{#1}\else{#2}\fi}%
\usepackage{amsmath,amsthm,amssymb}
\usepackage{amsfonts}
\usepackage[mathscr]{eucal}
\usepackage{ascmac}
\usepackage{cases,color,bm,enumerate}
\usepackage[all]{xy}
\toukoudel{\usepackage[noadjust]{cite}}
\theoremstyle{theorem}
\newtheorem{theorem}{Theorem}[section]
\newtheorem{proposition}[theorem]{Proposition}
\newtheorem{lemma}[theorem]{Lemma}
\newtheorem{corollary}[theorem]{Corollary}
\newtheorem{fact}[theorem]{Fact}
\theoremstyle{definition}
\newtheorem{definition}[theorem]{Definition}
\newtheorem{example}[theorem]{Example}
\theoremstyle{remark}
\newtheorem{remark}[theorem]{Remark}
\newtheorem*{acknowledgements}{Acknowledgements}

\allowdisplaybreaks
\newcommand{\inner}[2]{\left\langle{#1},{#2}\right\rangle}
\newcommand{\rank}{\operatorname{rank}}

\newcommand{\hess}{\operatorname{Hess}}

\newcommand{\R}{\bm{R}}

\newcommand{\e}{\bm{e}}
\newcommand{\n}{\bm{n}}
\newcommand{\F}{\mathcal{F}}

\newcommand{\vphi}{\varphi}
\newcommand{\eps}{\varepsilon}

\newcommand{\til}{\tilde}
\newcommand{\wtil}{\widetilde}

\renewcommand{\u}{\bm{u}}
\newcommand{\x}{\bm{x}}
\newcommand{\pmt}[1]{\begin{pmatrix}#1\end{pmatrix}}

\numberwithin{equation}{section}
\begin{document}

\title[Principal curvatures of frontals]
{Behavior of principal curvatures of frontals near non-front singular points and their application}

\toukouchange{
\author[K. Saji]{Kentaro Saji}
\address[K. Saji]{Department of Mathematics, Graduate School of Science, 
Kobe University, Rokko 1-1, Kobe 657-8501, Japan}
\email{saji@math.kobe-u.ac.jp}
\author[K. Teramoto]{Keisuke Teramoto}
\address[K. Teramoto]{Institute of Mathematics for Industry, Kyushu University, %
744 Motooka, Fukuoka 819-0395, Japan}
\email{k-teramoto@imi.kyushu-u.ac.jp}
}
{
\author{\name{Kentaro \surname{Saji}}}
\address{Department of Mathematics, Graduate School of Science, Kobe University, %
Rokko 1-1, Kobe 657-8501, Japan \email{saji@math.kobe-u.ac.jp}}
\author{\name{Keisuke \surname{Teramoto}}}
\address{Institute of Mathematics for Industry, Kyushu University, %
744 Motooka, Fukuoka 819-0395, Japan \email{k-teramoto@imi.kyushu-u.ac.jp}}
}

\toukouchange{}
{\classification{57R45; 53A05; 53A55}}

\thanks{The authors were partially supported by JSPS KAKENHI Grant Number JP18K03301 and JP19K14533, and
CAPES/JSPS Bilateral Joint Research Project}
\date{\today}

\begin{abstract}
We investigate behavior of principal curvatures 
and principal vectors near a non-degenerate singular point 
of the first kind of frontals.
As an application, 
we extend the notion of
Ribaucour transformations to frontals with singular points.
\end{abstract}

\maketitle

\section{Introduction}
In this paper, we study behavior of principal curvature
near a singular point
of frontal surfaces which is not a front.
A frontal is a class of surfaces with singular points,
and it is well known that surfaces with 
constant curvature are in this class.
In these decades,
there are several studies of frontals 
from the viewpoint of differential geometry 
and various geometric invariants at singular points 
are introduced 
\cite{maxface,framed,intrinsic-frontal,honda-saji,martins-saji,msuy,front}. 
It is known that a cuspidal edge 
($\mathcal{A}$-equivalent to the germ $(u,v)\mapsto(u,v^2,v^3)$ 
at the origin)
and a swallowtail are generic singularities of fronts in $3$-space. 
On the other hand, a cuspidal cross cap and a 
$5/2$-cuspidal edge are typical singularities of frontals
which are not front. 
Boundedness of Gaussian and mean curvature of frontals
at certain singular points are studied by
in terms of
geometric invariants \cite{honda-saji,msuy,front}.

Behavior of principal curvatures of fronts 
are studied in \cite{tera0,tera}.
Since singularities of frontals which are not front one
are a kind of degenerate 
singularities of fronts, it is natural to expect
principal curvatures differently behave.
We divide non-degenerate singular points of 
frontals which are not a front one into two classes: a singular point of
{\it $k$-non-front} and a singular point of {\it pure-frontal}. 
Typical examples of singular points of $k$-non-front are
cuspidal cross caps $(k=1)$ and cuspidal $S_{k-1}$ singular points,
and that of a singular point of pure-frontal is
a $5/2$-cuspidal edge.
Using geometric invariants,
we give a necessary and sufficient condition 
that the principal curvatures can be extended as 
$C^\infty$ functions
near a singular point of pure-frontal (Theorem \ref{thm:principal1}).
On the other hand, we show
around a singular point of $2k$-non-front,
one principal curvature can be extended as a continuous function.
We also show around a singular point of $(2k+1)$-non-front, 
on the singular curve $\gamma:(-\eps,\eps)\to U$ $(\gamma(0)=p)$,
one principal curvature can be extended as 
a continuous function across $\gamma((-\eps,0))$
and
the other principal curvature can be extended as 
a continuous function across $\gamma((0,\eps))$ 
(Theorem \ref{thm:principal-k-frontal}).

Moreover, we consider umbilic points of 
a frontal singular point which is not a front one 
(Sections \ref{sec:umbpure} and \ref{sec:umbknon}).
Furthermore, we study behavior of principal vector field,
and extend this notion to frontals as `curvature line frame' (Section
\ref{sec:curvlineframe}).
As an application, by using curvature line frame
we extend the notion of
Ribaucour transformations to frontals which
is given and studied for regular surfaces (Section \ref{sec:rib}).

\section{Preliminaries} 
We recall some notions and properties of frontals. 
\subsection{Frontals}
Let $f\colon V\to\R^3$ be a $C^\infty$ map, where $V$ is an open set of $\R^2$. 
Then $f$ is a {\it frontal} if there exists a $C^\infty$ map 
$\nu\colon V\to S^2$ 
such that $\inner{df_q(X)}{\nu(q)}=0$ holds for any $q\in V$ and $X\in T_q V$, 
where $S^2$ is the unit sphere in $\R^3$ and $\inner{\cdot}{\cdot}$ is the 
Euclidean inner product of $\R^3$. 
We call the map $\nu$ a {\it unit normal vector} or the {\it Gauss map} of $f$. 
If a frontal $f$ satisfies that the pair $(f,\nu)\colon V\to\R^3\times S^2$ is an immersion, 
then $f$ is called a {\it front}. 
We fix a frontal $f$. 
A point $p\in V$ is said to be a {\it singular point} of $f$ if $f$ is not an immersion at $p$. 
We denote by $S(f)$ the set of singular points of $f$ (on $V$). 
Let us set a function $\lambda\colon V\to\R$ by 
\begin{equation}\label{eq:lambda}
\lambda(u,v)=\det(f_u,f_v,\nu)(u,v),
\end{equation}
where $(u,v)$ are some coordinates, $(\cdot)_u=\partial/\partial u$ and $(\cdot)_v=\partial/\partial v$. 
We call $\lambda$ the {\it signed area density function} of $f$. 
A non-zero functional multiple of $\lambda$ is called 
an {\it identifier of singularities}.
Taking a singular point $p\in S(f)$ of a frontal $f$, 
$p$ is said to be {\it non-degenerate} 
if $(\hat{\lambda}_u(p),\hat{\lambda}_v(p))\neq(0,0)$,
where $\hat\lambda$ is an identifier of singularities.
We notice that $\rank df_p=1$ if $p\in S(f)$ is non-degenerate.
If $p\in S(f)$ is non-degenerate,
then there exist a neighborhood $U$ of $p$ and a $C^\infty$ regular curve 
$\gamma=\gamma(t)\colon (-\eps,\eps)\to U$ ($\eps>0$) such that 
$\gamma(0)=p$ and $\hat{\lambda}(\gamma(t))=0$ holds. 
This implies that $S(f)$ is locally parametrized by $\gamma$. 
Moreover, there exists a non-zero vector field $\eta$ on $U$ such that 
$df_q(\eta_q)=0$ for any $q\in S(f)\cap U$. 
We call $\gamma$ and $\eta$ a {\it singular curve} and a {\it null vector field}, respectively. 
A non-degenerate singular point $p$ is said to be of the {\it first kind\/}
if $\eta$ is transverse to $\gamma$ at $p$.

Let $f\colon V\to\R^3$ be a frontal, and let $p\in S(f)$ be non-degenerate.
We set $\nu$ a unit normal vector of $f$. 
Let $\gamma(t)$ be a singular curve through $p$ and $\eta$ a null vector field. 
Then we define two functions $\delta$ and $\psi$ by 
\begin{equation}\label{eq:delta}
\delta(t)=\det(\gamma'(t),\eta(t)),\quad
\psi(t)=\det(\hat{\gamma}'(t),\nu(\gamma(t)),\eta\nu(\gamma(t))),
\end{equation}
where $'=d/dt$, $\hat{\gamma}=f\circ\gamma$. 
By definition, $p$ is of the first kind if and only if $\delta(0)\neq0$. 
Moreover, it is known that a frontal $f$ is not a front at a singular point of the first kind 
if and only if $\psi(0)=0$. A singular 
point of the first kind $p$ of a frontal $f$ is 
said to be {\it non-front} 
if $\psi(p)=0$, namely, $f$ is not a front at $p$.
We divide non-front singular points as follows:

\begin{definition}
\begin{enumerate}
\item A singular point of the first kind $p$ of a frontal $f$ 
is said to be a {\it $k$-non-front singular point} $(k\geq1)$ 
if the function $\psi$ as in \eqref{eq:delta} satisfies 
$\psi(0)=\psi'(0)=\cdots=\psi^{(k-1)}(0)=0$ and $\psi^{(k)}(0)\neq0$. 
\item A singular point of the first kind $p$ of 
a frontal $f$ is a {\it pure-frontal singular point} 
if the function $\psi$ vanishes identically along $\gamma(t)$. 
\end{enumerate}
\end{definition}
A cuspidal edge is a front.
Typical examples of singular points of $k$-non-front are cuspidal 
$S_k^\pm$ singularities $(k\geq0)$ which are 
$\mathcal{A}$-equivalent to the germ $(u,v)\mapsto(u,v^2,v^3(u^{k+1}\pm v^2))$ at the origin. 
Two map-germs $f_1,f_2$ are {\it $\mathcal{A}$-equivalent\/}
if they coincide up to coordinate transformations of the
source and the target spaces.
On the other hand, a typical example of a singular point of pure-frontal
is a $5/2$-cuspidal edge 
which is $\mathcal{A}$-equivalent to the germ $(u,v)\mapsto(u,v^2,v^5)$ (see Figure \ref{fig:sing1}). 
For criteria and geometric properties of surfaces with these singularities, see \cite{maxface,honda-koiso-saji,honda-saji,intrinsic-frontal,oset-saji,msuy,saji-csk}.
\begin{figure}[htbp]
  \begin{center}
    \begin{tabular}{c}
          \includegraphics[width=.25\linewidth]{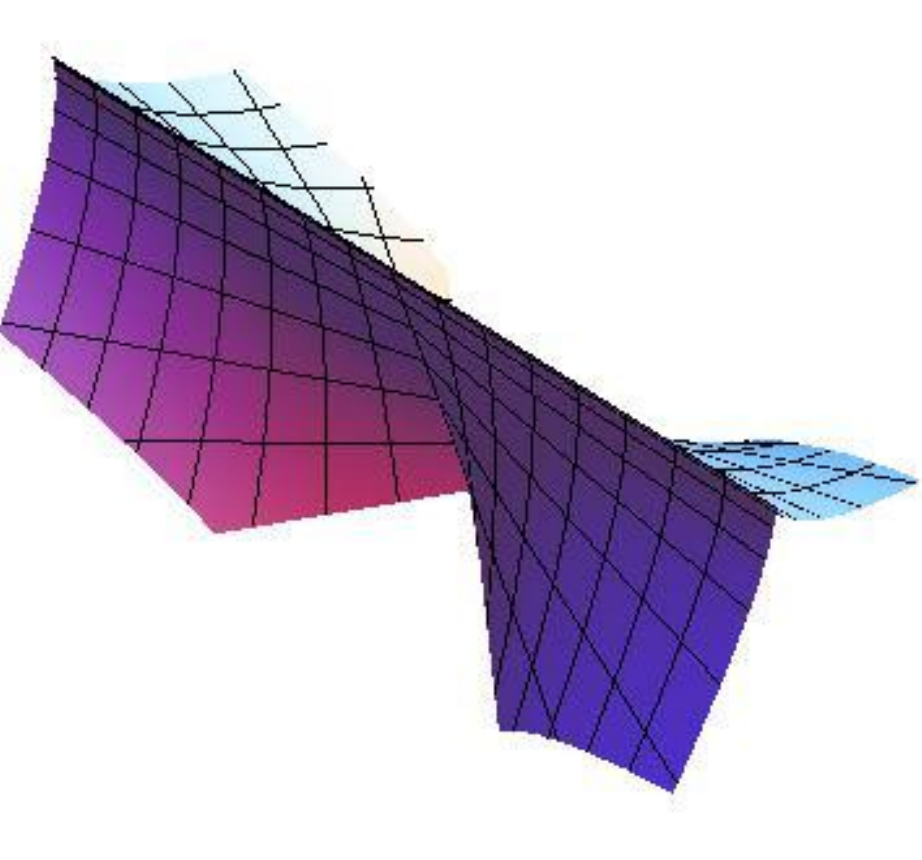}
          \includegraphics[width=.25\linewidth]{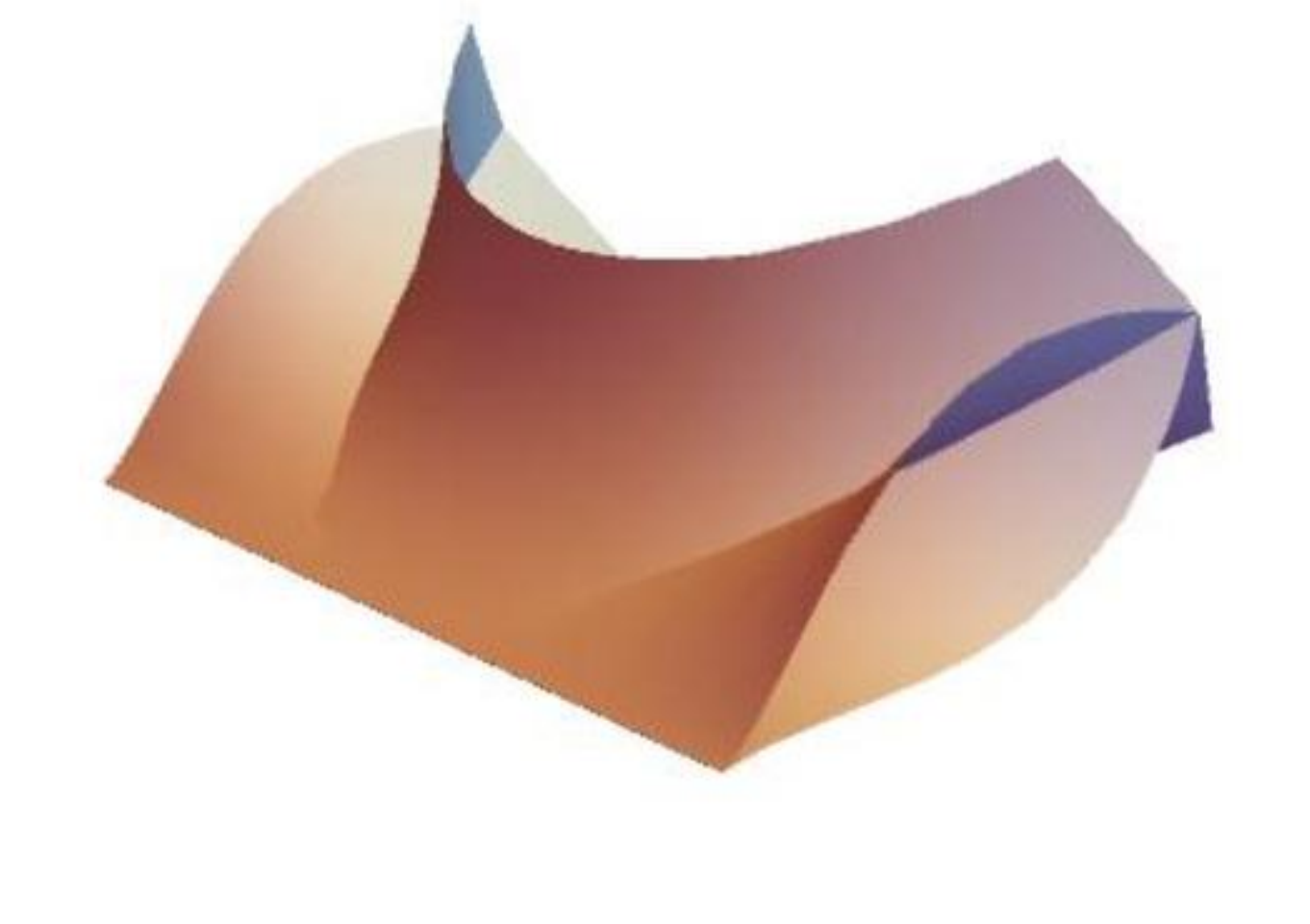}
          \includegraphics[width=.25\linewidth]{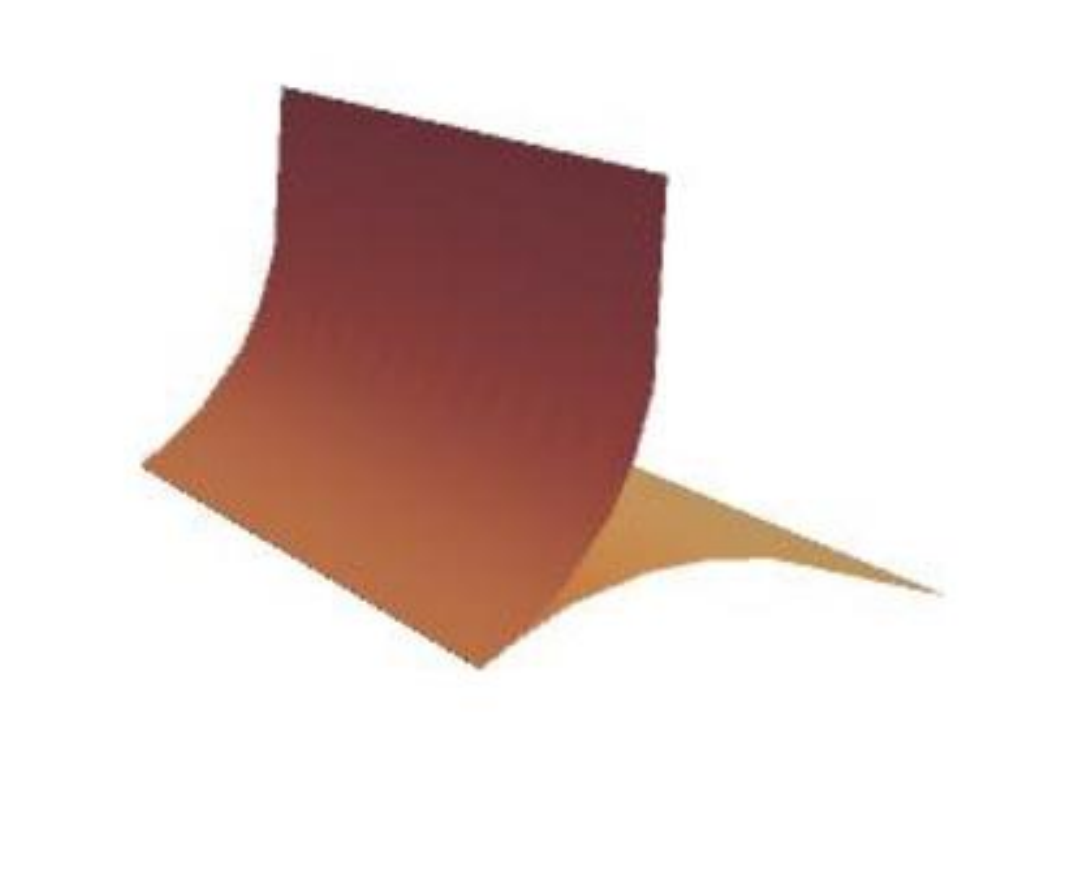}
    \end{tabular}
    \caption{From left to right: cuspidal $S_0$ singularity (cuspidal cross cap), cuspidal $S_1^-$ singularity and $5/2$-cuspidal edge.}
    \label{fig:sing1}
  \end{center}
\end{figure}
Let $f$ be a frontal and $p$ a singular point of the first kind. 
Then there are several differential geometric invariants at $p$.
We introduce here {\it singular curvature} $\kappa_s$,
the {\it limiting normal curvature} $\kappa_\nu$,
the {\it cuspidal curvature} $\kappa_c$ and
the {\it cuspidal torsion} $\kappa_t$,
where the presise definition themselves will not needed 
(See Lemma \ref{prop:invariants}.).
See \cite{front} for $\kappa_s,\kappa_\nu$,
and \cite{martins-saji,msuy} for the others.
We remark although these invariants were defined 
cuspidal edge singular point
it is a singular point of the first kind,
one can easily see these definitions work our case.
For a frontal $f$, the following assertion is known.
\begin{fact}[\cite{msuy}]\label{fact:frontal}
Let $f\colon V\to\R^3$ be a frontal and $p$ 
a singular point of the first kind. 
Then $p$ is a non-front singular point of $f$ 
if and only if $\kappa_c(p)=0$ holds. 
In particular, a singular curve $\gamma(t)$ through 
$p$ consists of pure-frontal singularities 
if and only if $\kappa_c$ vanishes along $\gamma(t)$.
\end{fact}

One can take a pair of positively oriented vector fields 
$(\xi,\eta)$ on a 
neighborhood $U$ of a singular point of the first kind 
$p$ satisfying that
$\xi$ is tangent to $\gamma$, and
$\eta$ is a null vector field.
We call such a pair $(\xi,\eta)$ an 
{\it adapted pair of vector field} (\cite{martins-saji}). 
On the other hand, 
a local coordinate system $(u,v)$ on $U$
satisfying that
the $u$-axis coincides with the image of singular curve, 
and $(\partial_u,\partial_v)$ is an adapted pair
is said to be {\it adapted\/} (\cite{msuy,front}). 

If a point $p$ is a non-front singular point of a frontal $f$, 
one can take an adapted pair $(\xi,\tilde\eta)$
of vector fields such that
$\inner{\tilde{\eta}\tilde{\eta} f(p)}{\xi f(p)}=\inner{\tilde{\eta}\tilde{\eta}\tilde{\eta} f(p)}{\xi f(p)}=0$.  
Thus there exists a number $l\in \R$ such that $\tilde{\eta}\tilde{\eta}\tilde{\eta} f(p)=l\tilde{\eta}\tilde{\eta} f(p)$. 
Using this null vector field $\tilde{\eta}$ and the number $l$, we set other invariant for a frontal of the first kind.
\begin{align}
\begin{aligned}\label{eq:invariants2}
r_b(p)&=\dfrac{|\xi f|^2\det(\xi f,\tilde{\eta}\tilde{\eta} f,\tilde{\eta}^4 f)}
{|\xi f\times\tilde{\eta}\tilde{\eta}f|^3}(p),\\
r_c(p)&=\dfrac{|\xi f|^{5/2}\det(\xi f,\tilde{\eta}\tilde{\eta} f,3\tilde{\eta}^5 f-10l\tilde{\eta}^4f)}
{|\xi f\times \tilde{\eta}\tilde{\eta}f|^{7/2}}(p),
\end{aligned}
\end{align}
where $\tilde{\eta}^{k}f$ means $k$-times directional derivative of $f$ in the direction $\tilde{\eta}$. 
We call $r_b(p)$ and $r_c(p)$ the {\it bias} and the {\it secondary cuspidal curvature} of $f$ at $p$, respectively (\cite{oset-saji}). 

Let $p$ be a pure-frontal singular point of $f$.
there exists a function $l\colon(-\eps,\eps)\to\R$ such that 
$\tilde{\eta}\tilde{\eta}\tilde{\eta}f(\gamma(t))=l(t)\tilde{\eta}\tilde{\eta}f(\gamma(t))$. 
Thus for a pure-frontal singular point, 
we can define $r_b$ and $r_c$ along $\gamma(t)$ 
by 
$r_b(t)=r_b(\gamma(t))$ and $r_c(t)=r_c(\gamma(t))$.
We also call $r_b(t)$ and $r_c(t)$ the {\it bias} and the 
{\it secondary cuspidal curvature} along $\gamma(t)$, 
respectively (\cite{honda-saji}).

\begin{fact}\label{fact:u-orthogonal}
Let $f\colon V\to\R^3$ be a frontal, and let $p$ be
a singular point of the first kind. 
There exists an adapted coordinate system $(u,v)$ around $p$ such that 
{\rm (1)}
$|f_u(u,0)|=|f_{vv}(u,0)|=1$,
$\inner{f_u}{f_{vv}}(u,0)=0$ or
{\rm (2)}
$\inner{f_u}{f_{vv}}(u,0)=\inner{f_u}{f_{vvv}}(u,0)=0.$
\end{fact}
An adapted coordinate system satisfying (1) (resp. (2)) 
is said to be 
{\it orthogonal adapted} (resp. {\it normally adapted}).
See {\cite[Corollary 3.5]{honda-saji}} for a proof of (2).
The other statements can be shown by the similar way.

\subsection{Fundamental forms and invariants} 
We consider coefficients of the first and the second fundamental form of a frontal. 
Let $f\colon V\to\R^3$ be a frontal, and $p$ be a non-front singular point. 
Then we take an adapted coordinate system $(U;u,v)$ centered at $p$. 
Since $f_v(u,0)=0$, there exists a $C^\infty$ function $h\colon U\to\R^3$ such that $f_v=vh$ by the division lemma. 
Moreover, since any $(u,0)\in S(f)$ is a 
singular point of the first kind, 
$\lambda_v(u,0)=\det(f_u,h,\nu)(u,0)\neq0$. 
Thus $h\neq0$ near $p$ and $\{f_u,h,\nu\}$ forms a frame along $f$. 
Using this frame, we define the following functions on $U$:
\begin{align}
\begin{aligned}\label{eq:fundamentals}
\tilde{E}&=\inner{f_u}{f_u},& \tilde{F}&=\inner{f_u}{h}, & \tilde{G}&=\inner{h}{h},\\
\tilde{L}&=-\inner{f_u}{\nu_u},& \tilde{M}&=-\inner{h}{\nu_u}, & \tilde{N}&=-\inner{h}{\nu_v}.
\end{aligned}
\end{align}
We note that $\tilde{E}\tilde{G}-\tilde{F}^2>0$ on $U$. 
Moreover, we notice that $\nu$ can be chosen as $\nu=\pm(f_u\times h)/|f_u\times h|$. 

Let us denote by $E$, $F$, $G$, $L$, $M$ and $N$ the coefficients of the first and the second fundamental form of $f$ on $U\setminus\{v=0\}$
obtained by the usual manner. 
Then we have 
\begin{equation}\label{eq:difference}
E=\tilde{E},\ F=v\tilde{F},\ G=v^2\tilde{G},\ 
L=\tilde{L},\ M=v\tilde{M},\ N=v\tilde{N}.
\end{equation}

\begin{lemma}\label{lem:tilde-eta}
Let $f$ be a frontal and $p$ a non-front singular point. 
Take an orthogonal adapted coordinate system $(U;u,v)$ around $p$. 
Then we can take a null vector field $\tilde{\eta}$ satisfying 
$\inner{f_u}{\tilde{\eta}\tilde{\eta}f}=\inner{f_u}{\tilde{\eta}\tilde{\eta}\tilde{\eta}f}=0$ at $p$ by setting 
\begin{equation}\label{eq:tilde-eta}
\tilde{\eta}=-v^2\tilde{F}_v(p)\partial_u+\partial_v.
\end{equation}
\end{lemma}
\begin{proof}
Taking an orthogonal adapted coordinate system, 
an adapted pair of vector fields $(\xi,\eta)$ is given by $(\xi,\eta)=(\partial_u,\partial_v)$.  
We set $\tilde{\eta}$ as 
$\tilde{\eta}=a(u,v)\partial_u+\partial_v$
on $U$. 
Since $\tilde{\eta}$ is also a null vector field of $f$, 
$\tilde{\eta}f=af_u+vh=0$ on the $u$-axis, where $f_v=vh$. 
Thus $a(u,0)=0$ holds, in particular, $a(p)=0$. 

We next consider the second and the third order directional derivatives 
of $f$ in the direction $\tilde{\eta}$. 
By a direct calculation, it follows that 
\begin{align*}
\tilde{\eta}\tilde{\eta}f&=a(\tilde{\eta}f)_u+a_vf_u+vah_u+h+vh_v,\\
\tilde{\eta}\tilde{\eta}\tilde{\eta}f&
=a(\tilde{\eta}\tilde{\eta}f)_u+a_v(\tilde{\eta}f)_u+a(\tilde{\eta}f)_{uv}+a_{vv}f_u+2va_vh_u+ah_{uv}+2h_v+vh_{vv}.
\end{align*}
Since $a(u,0)=0$, $f_{vv}(u,0)=h(u,0)$, $\inner{f_u}{f_u}(u,0)=1$ and $\inner{f_u}{f_{vv}}(u,0)=0$, 
we see that $\inner{\xi f}{\tilde{\eta}\tilde{\eta}f}=\inner{f_u}{\tilde{\eta}\tilde{\eta}f}=a_v$ holds along the $u$-axis. 
Thus we get $a_v(p)=0$. 
Under this assumption, we have 
$\inner{\xi f}{\tilde{\eta}\tilde{\eta}\tilde{\eta}f}=\inner{f_u}{\tilde{\eta}\tilde{\eta}\tilde{\eta}f}=a_{vv}+2\inner{f_u}{h_v}$ at $p$. 
By the definition of $\tilde{F}$, we see
$\tilde{F}_v=\inner{f_u}{h}_v=\inner{f_{uv}}{h}+\inner{f_u}{h_v}=\inner{f_u}{h_v}$ along the $u$-axis. 
Especially, $\tilde{F}_v(p)=\inner{f_u}{h_v}(p)$ holds. 
Thus setting $a_{vv}(p)=-2\tilde{F}_v(p)$, $\inner{\xi f}{\tilde{\eta}\tilde{\eta}\tilde{\eta}f}=0$ at $p$, 
and hence we have the assertion.  
\end{proof}

By Lemma \ref{lem:tilde-eta}, 
if $p$ is a pure-frontal singular point of $f$, then we can take $\tilde{\eta}$ as 
\begin{equation}\label{eq:pure-eta}
\tilde{\eta}=-v^2\tilde{F}_v(u,0)\partial_u+\partial_v
\end{equation}
on $U$. 

\begin{lemma}\label{prop:invariants}
Let $f\colon U\to\R^3$ be a frontal and $p\in U$ a non-front singular point of $f$. 
Let $(u,v)$ be an orthogonal adapted coordinate system around $p$ 
satisfying that $\det(f_u,f_{vv},\nu)(u,0)>0$. 
Then 
\begin{align}
\begin{aligned}\label{eq:invariants}
\kappa_\nu(u)&={\tilde{L}}(u,0),\quad 
\kappa_c(u)={2\tilde{N}}(u,0),\quad
\kappa_t(u)={\tilde{M}}(u,0),\\
r_b(p)&={3\tilde{N}_v}(p),\quad
r_c(p)=12\left(\tilde{N}_{vv}-4\tilde{F}_v\tilde{M}-{2\tilde{G}_v\tilde{N}_v}\right)(p)
\end{aligned}
\end{align}
hold.
\end{lemma}
\begin{proof}
For $\kappa_\nu$, $\kappa_c$ and $\kappa_t$, see \cite[Lemma 2.7]{tera}. 
We show $r_b$ and $r_c$. 
First, we calculate directional derivatives of $f$ in the direction $\tilde{\eta}$ as in \eqref{eq:tilde-eta}. 
By the proof of Lemma \ref{lem:tilde-eta}, we see that 
$\tilde{\eta}^3f=a_{vv}f_u+2h_v=\tilde{G}_vh=\tilde{G}_v\tilde{\eta}^2f$ holds at $p$, where $a=-v^2\tilde{F}_v(p)$. 
In particular, we have $l=\tilde{G}_v(p)$ and 
\begin{equation}\label{eq:fvfvgv}
h_v(p)=\tilde{F}_v(p)f_u(p)+(\tilde{G}_v(p)/2)h(p).
\end{equation}
We consider $\tilde{\eta}^4f$ and $\tilde{\eta}^5f$. 
By direct calculations, we have 
\begin{align*}
\tilde{\eta}^4f&=a(\tilde{\eta}^3f)_u+a_v(\tilde{\eta}^2f)_u+a(\tilde{\eta}^2f)_{uv}
+a_{vv}(\tilde{\eta}f)_u+2a_v(\tilde{\eta}f)_{uv}+a(\tilde{\eta}f)_{uvv}\\
&\quad +a_{vvv}f_u+3a_{vv}f_{uv}+3a_vf_{uvv}+af_{uvvv}+f_{vvvv},\\
\tilde{\eta}^5f&=a(\tilde{\eta}^4f)_u+a_v(\tilde{\eta}^3f)+a(\tilde{\eta}^3f)_{v}
+a_{vv}(\tilde{\eta}^2f)_u+2a_v(\tilde{\eta}^2f)_{uv}+a(\tilde{\eta}^2f)_{uvv}\\
&\quad+a_{vvv}(\tilde{\eta}f)_u+3a_{vv}(\tilde{\eta}f)_{uv}+3a_v(\tilde{\eta}f)_{uvv}+a(\tilde{\eta}f)_{uvvv}\\
&\quad +a_{vvvv}f_u+4a_{vvv}f_{uv}+6a_{vv}f_{uvv}+4a_{v}f_{uvvv}+af_{uvvvv}+f_{vvvvv}.
\end{align*}
Since $\tilde{\eta}f=af_u+f_v$ and $\tilde{\eta}^2f=a(\tilde{\eta}f)_u+a_vf_u+af_{uv}+f_{vv}$, 
we get $(\tilde{\eta}f)_{uv}=f_{uvv}$ and $(\tilde{\eta}^2f)_u=f_{uvv}$ at $p$ because 
$a(p)=a_v(p)=f_{uv}(p)=(\tilde{\eta}f)_u(p)=0$. 
Thus we see that 
$\tilde{\eta}^4f=f_{vvvv}$ and 
$\tilde{\eta}^5f=10a_{vv}f_{uvv}+f_{vvvvv}$ 
hold at $p$. 
Since $f_v=vh$, we obtain 
$f_{uvv}=h_u$, $f_{vvvv}=3h_{vv}$ and $f_{vvvvv}=4h_{vvv}$ at $p$. 
Therefore it follows that 
\begin{equation}\label{eq:eta^4}
\tilde{\eta}^4f=3h_{vv},\quad 
\tilde{\eta}^5f=10a_{vv}h_u+4h_{vvv}
\end{equation}
at $p$. 

We consider $r_b$. 
Since $(u,v)$ is an orthogonal adapted coordinate system, 
we have $|f_u|^2=1$, $|h|^2=1$, $\inner{f_u}{h}=0$,  
$$\dfrac{\det(f_u,\tilde{\eta}^2f,\tilde{\eta}^4f)}{|f_u\times \tilde{\eta}^2f|^3}=\dfrac{3\inner{\nu}{h_{vv}}}{|f_u\times h|^2}=3\inner{\nu}{h_{vv}}$$
along the $u$-axis, 
where we use relations $\nu=(f_u\times h)/|f_u\times h|$ and \eqref{eq:eta^4}.  
Since $\inner{h}{\nu}=0$, we have
$\inner{h_v}{\nu}+\inner{h}{\nu_v}=0$, 
$\inner{h_{vv}}{\nu}+2\inner{h_v}{\nu_v}+\inner{h}{\nu_{vv}}=0$.
Further, by Fact \ref{fact:frontal} and the expression of 
$\kappa_c$ as in \eqref{eq:invariants}, 
$\tilde{N}=-\inner{h}{\nu_v}=0$ at $p$. 
Thus $\nu_v(p)=0$ holds, and hence we have 
$\inner{h_{vv}}{\nu}(p)+\inner{h}{\nu_{vv}}(p)=0$. 
On the other hand, differentiating $\tilde{N}=-\inner{h}{\nu_v}$ by $v$, we see that 
$\tilde{N}_v=-\inner{h_v}{\nu_v}-\inner{h}{\nu_{vv}}$. 
Since $\inner{h}{\nu_{vv}}(p)=-\inner{h_{vv}}{\nu}(p)$, it holds that 
$\tilde{N}_v(p)=\inner{h_{vv}}{\nu}(p)$. 
Therefore by \eqref{eq:invariants2}, $r_b$ can be expressed as in \eqref{eq:invariants}.

Finally, we consider $r_c$. 
By above calculations, \eqref{eq:invariants2} and \eqref{eq:eta^4},
we see that
$$r_c=30a_{vv}\inner{\nu}{h_u}+{\inner{\nu}{12h_{vvv}-30lh_{vv}}}=-60\tilde{F}_v\tilde{M}+\inner{\nu}{12h_{vvv}-30lh_{vv}}$$
holds at $p$, where $l=\tilde{G}_v(p)$. 
The first and the second order derivatives of $\tilde{N}=-\inner{h}{\nu_v}$ by $v$ are 
$\tilde{N}_v=-\inner{h}{\nu_{vv}}$ and $\tilde{N}_{vv}=-2\inner{h_v}{\nu_{vv}}-\inner{h}{\nu_{vvv}}$ at $p$. 
On the other hand, differentiating $\inner{h}{\nu}=0$ by $v$, we have 
$\inner{h_{vv}}{\nu}=-\inner{h}{\nu_{vv}}$, $\inner{h_{vvv}}{\nu}+3\inner{h_v}{\nu_{vv}}=-\inner{h}{\nu_{vvv}}$ at $p$.
Thus 
$$\inner{h_{vv}}{\nu}=-\inner{h}{\nu_{vv}}=\tilde{N}_v,\quad 
\inner{h_{vvv}}{\nu}=\tilde{N}_{vv}-\inner{h_{v}}{\nu_{vv}}.$$
Noticing that $h_v=\tilde{F}_{v}f_u+(\tilde{G}_v/2)h$ holds at $p$, we have 
$\inner{h_{vvv}}{\nu}=\tilde{N}_{vv}-\tilde{F}_v\inner{f_u}{\nu_{vv}}-(\tilde{G}_v/2)\inner{h}{\nu_{vv}}$. 
Differentiating $\inner{f_u}{\nu}=0$ by $v$ twice, we have 
$\inner{f_{uvv}}{\nu}+\inner{f_{u}}{\nu_{vv}}=\inner{h_u}{\nu}+\inner{f_{u}}{\nu_{vv}}=\tilde{M}+\inner{f_{u}}{\nu_{vv}}=0$. 
Thus it holds that 
$\inner{h_{vvv}}{\nu}=N_{vv}+\tilde{F}_v\tilde{M}+(\tilde{G}_v/2)\tilde{N}_v$ at $p$, 
and hence
$$r_c=-60\tilde{F}_v\tilde{M}+\left(12\inner{\nu}{h_{vvv}}-{30\tilde{G}_v}\inner{\nu}{h_{vv}}\right)
={12}\left(\tilde{N}_{vv}-4\tilde{F}_v\tilde{M}-{2\tilde{G}_v\tilde{N}_v}\right)$$
holds at $p$.
\end{proof}
\begin{corollary}\label{cor:25-invariants}
Under the same assumptions as in Lemma \ref{prop:invariants}, 
if $f$ is a frontal and $p$ is a pure-frontal singular point, then the bias $r_b$ and the secondary cuspidal curvature $r_c$ are written as 
\begin{equation}\label{eq:25invariants}
r_b(u)={3\tilde{N}_v}(u,0),\quad
r_c(u)=12\left(\tilde{N}_{vv}-4\tilde{F}_v\tilde{M}-{2\tilde{G}_v\tilde{N}_v}\right)(u,0)
\end{equation} 
along the $u$-axis.
\end{corollary}
\begin{proof}
We take an orthogonal adapted coordinate system $(U;u,v)$ around $p$ with $\det(f_u,h,\nu)(u,0)>0$. 
Since $\kappa_c(u)=0$ by Fact \ref{fact:frontal}, $\tilde{N}(u,0)=-\inner{h}{\nu_v}(u,0)=0$, 
especially $\nu_v(u,0)=0$. 
By similar calculations with $\tilde{\eta}$ as in \eqref{eq:pure-eta}, we have the assertions.
\end{proof}

For the Gauss map $\nu$ of a frontal $f$, by direct calculations
we have the following.
\begin{lemma}[{cf. \cite[Lemma 2.1]{tera0}}]\label{lem:weingarten}
Let $(U;u,v)$ be an adapted coordinate system  around a non-degenerate singular point of the first kind 
of a frontal $f$. 
Then 
\begin{equation}\label{eq:weingarten}
\nu_u=\frac{\tilde{F}\tilde{M}-\tilde{G}\tilde{L}}{\tilde{E}\tilde{G}-\tilde{F}^2}f_u+
\frac{\tilde{F}\tilde{L}-\tilde{E}\tilde{M}}{\tilde{E}\tilde{G}-\tilde{F}^2}h,\ 
\nu_v=\frac{\tilde{F}\tilde{N}-v\tilde{G}\tilde{M}}{\tilde{E}\tilde{G}-\tilde{F}^2}f_u+
\frac{v\tilde{F}\tilde{M}-\tilde{E}\tilde{N}}{\tilde{E}\tilde{G}-\tilde{F}^2}h.
\end{equation}
\end{lemma}

The Gaussian curvature $K$ and the mean curvature $H$ 
of a frontal $f$ with a singular point of the first kind are given as 
\begin{equation}\label{eq:K-H}
K=\dfrac{\tilde{L}\tilde{N}-v\tilde{M}^2}{v(\tilde{E}\tilde{G}-\tilde{F}^2)},\quad
H=\dfrac{\tilde{E}\tilde{N}-2v\tilde{F}\tilde{M}+v\tilde{G}\tilde{L}}{2v(\tilde{E}\tilde{G}-\tilde{F}^2)}.
\end{equation}
We note that behavior of $K$ and $H$ are 
investigated in \cite{msuy,front}. 
We here observe the behavior of the function $\Gamma=H^2-K$.
By a direct computation, 
\begin{align}
\label{eq:umbilicity1}
 &4v^2(\tilde{E}\tilde{G}-\tilde{F}^2)^2\Gamma\nonumber\\
=&
(\tilde{E}\tilde{N}-2v\tilde{F}\tilde{M}+v\tilde{G}\tilde{L})^2-4v(\tilde{E}\tilde{G}-\tilde{F}^2)(\tilde{L}\tilde{N}-v\tilde{M}^2)\\
\label{eq:umbilicity2}
=&
4v^2\dfrac{
(\tilde{E}\tilde{G}-\tilde{F}^2)(\tilde{E}\tilde{M}-\tilde{F}\tilde{L})^2}
{\tilde{E}}
+\left((\tilde{E}\tilde{N}-v\tilde{G}\tilde{L})-\dfrac{2v\tilde{F}}{\tilde{E}}(\tilde{E}\tilde{M}-\tilde{F}\tilde{L})\right)^2.
\end{align}
holds on $U\setminus\{v=0\}$. 
In particular, $\Gamma\geq0$.

\section{Principal curvatures near a pure-frontal singular point}
We consider principal curvatures of a frontal near a pure-frontal singular point. 
\subsection{Behavior of principal curvatures near a pure-frontal singular point}
Let $f\colon V\to\R^3$ be a frontal and 
$p$ a pure-frontal singular point. 
Then we take an orthogonal adapted coordinate system 
$(u,v)$ on $U$ centered at $p$.
By Fact \ref{fact:frontal}, \eqref{eq:invariants} and \eqref{eq:weingarten}, 
$\nu_v=0$ holds along the $u$-axis. 
Thus there exists a map $\vphi\colon U\to\R^3$ $(U\subset V)$ such that $\nu_v=v\vphi$. 
Moreover, since $\tilde{N}(u,0)=0$ by Fact \ref{fact:frontal} and \eqref{eq:invariants}, 
there exists a function $\tilde{N}_1$ such that $\tilde{N}=v\tilde{N}_1$, 
where $\tilde{N}_1=-\inner{h}{\vphi}$ holds. 

We notice that the Gaussian curvature 
and the mean curvature can be written as 
\begin{equation}\label{eq:K-H1}
K=\dfrac{\tilde{L}\tilde{N}_1-\tilde{M}^2}
{\tilde{E}\tilde{G}-\tilde{F}^2},\quad
H=\dfrac{\tilde{E}\tilde{N}_1
-2\tilde{F}\tilde{M}+\tilde{G}\tilde{L}}{2(\tilde{E}\tilde{G}-\tilde{F}^2)}.
\end{equation}
Thus both $K$ and $H$ are bounded $C^\infty$ functions on $U$. 
We have
\begin{equation}\label{eq:umbilicity-25}
\Gamma=\dfrac{(\tilde{E}\tilde{N}_1+\tilde{G}\tilde{L})^2
-4\tilde{E}\tilde{G}(\tilde{L}\tilde{N}_1-\tilde{M}^2)}
{4\tilde{E}^2\tilde{G}^2}
=\dfrac{1}{4}\left({\tilde{N}_1}
-{\tilde{L}}\right)^2
+{\tilde{M}^2}\geq0
\end{equation}
along the $u$-axis by \eqref{eq:umbilicity1} and \eqref{eq:K-H1}. 
Since $H$ and $K$ are $C^\infty$ functions, $\Gamma$ is also a $C^\infty$ function on $U$. 
Using $K$ and $H$, we define two functions 
$\kappa_j$ $(j=1,2)$ on $U$ by 
\begin{equation}\label{eq:25-principal}
\kappa_1=H+\sqrt{H^2-K},\quad \kappa_2=H-\sqrt{H^2-K}.
\end{equation}
These functions satisfy $\kappa_1\kappa_2=K$ and 
$\kappa_1+\kappa_2=2H$. 
As we will see later, we may regard $\kappa_1$ and 
$\kappa_2$ as principal curvatures of $f$. 
Since $H^2-K\geq0$, $\kappa_1$ and $\kappa_2$ 
are continuous functions on $U$. 
A point $p\in S(f)$ is an {\it umbilic point} of $f$ 
if $\Gamma(p)=0$. 
By \eqref{eq:25-principal}, if $p$ is an umbilic point,
then $\kappa_1=\kappa_2$ at $p$.

\begin{theorem}\label{thm:principal1}
Let $f$ be a frontal and $p$ a pure-frontal singular point. 
Then principal curvatures $\kappa_j$ $(j=1,2)$ of 
$f$ can be extended as $C^\infty$ functions near $p$ if and only if 
\begin{equation}\label{eq:umbilicity4}
\left(\dfrac{1}{3}r_b-\kappa_\nu\right)^2+4\kappa_t^2\neq0
\end{equation}
holds along the singular curve through $p$. 
In particular, a point $p$ is an umbilic point of 
$f$ if and only if $r_b(p)=3\kappa_\nu(p)$ and $\kappa_t(p)=0$ hold.
\end{theorem}
\begin{proof}
Let us take an orthogonal adapted coordinate system $(U;u,v)$ centered at $p$. 
Then by the assumption, we see that $\tilde{N}=v\tilde{N}_1$, and hence $\tilde{N}_v=\tilde{N}_1$ holds along the $u$-axis. 
By \eqref{eq:K-H1}, \eqref{eq:umbilicity-25}, \eqref{eq:invariants}, 
Lemma \ref{prop:invariants} and Corollary \ref{cor:25-invariants}, 
we see that 
\begin{equation}\label{eq:hgamma}
H=\dfrac{\kappa_\nu}{2}+\dfrac{r_b}{6},\quad 
\Gamma=H^2-K=\dfrac{1}{4}\left(\dfrac{r_b}{3}-\kappa_\nu\right)^2+\kappa_t^2
\end{equation}
holds at $p$ (cf. \cite{honda-saji}). 
Since
$\kappa_j=H+(-1)^{j+1}\sqrt{\Gamma}$,
we see the assertion.
The last asserton is obvious by \eqref{eq:hgamma}.
\end{proof}

\subsection{Principal vectors}
Let $f\colon V\to\R^3$ be a frontal and 
$p$ a singular point of the first kind. 
We assume that $p$ is a pure-frontal singular point of $f$,
and not an umbilic point.
We consider the principal vectors.

A vector $\bm{V}=(V_1,V_2)=V_1\partial_u+V_2\partial_v\in T_pU$ 
is a {\it principal vector
relative to\/} $\kappa$ if 
$$
\kappa\pmt{E&F\\F&G}\pmt{V_1\\V_2}
=
\pmt{L&M\\M&N}\pmt{V_1\\V_2}
$$
for the coefficients of the first and the second
fundamental forms.
The number $\kappa$ satisfying the above is called
the {\it principal curvature}, and also satisfies
\eqref{eq:25-principal}.
By Lemma \ref{lem:weingarten},
if $\bm{V}_j=(V_1^j,V_2^j)$ $(j=1,2)$ are principal vectors
relative to $\kappa_j$, then
\begin{equation}\label{eq:p-dir}
\begin{pmatrix}
\tilde{L}-\kappa_j\tilde{E} & v(\tilde{M}-\kappa_j\tilde{F}) \\
v(\tilde{M}-\kappa_j\tilde{F}) & v^2(\tilde{N}_1-\kappa_j\tilde{G})
\end{pmatrix}
\begin{pmatrix} V_1^j \\ V_2^j \end{pmatrix}=
\begin{pmatrix} 0 \\ 0 \end{pmatrix}
\end{equation} 
on $U$, where $\tilde{N}_1$ is a function satisfying $\tilde{N}=v\tilde{N}_1$. 
By factoring out $v$, the equation \eqref{eq:p-dir} are equivalent to 
\begin{equation}\label{eq:p-dir2}
\begin{pmatrix}
\tilde{L}-\kappa_j\tilde{E} & v(\tilde{M}-\kappa_j\tilde{F}) \\
\tilde{M}-\kappa_j\tilde{F} & v(\tilde{N}_1-\kappa_j\tilde{G})
\end{pmatrix}
\begin{pmatrix} V_1^j \\ V_2^j \end{pmatrix}=
\begin{pmatrix} 0 \\ 0 \end{pmatrix}.
\end{equation}
The $(1,1)$ and $(2,1)$ elements of the matrix as in \eqref{eq:p-dir2} are 
\begin{equation}\label{eq:lkemkf}
\tilde{L}-\kappa_j\tilde{E}=\dfrac{1}{2}\left\{\left({\kappa_\nu}-\dfrac{r_b}{3}\right)+(-1)^j\sqrt{\left(\dfrac{r_b}{3}-\kappa_\nu\right)^2+4\kappa_t^2}\right\},\quad 
\tilde{M}-\kappa_j\tilde{F}=\kappa_t,
\end{equation}
and $\tilde N_1-\kappa_j\tilde G=-\tilde{L}+\kappa_{j+1}\tilde{E}$ 
(we think $\kappa_3=\kappa_1$) at $p$
 (we take $-$ sign if $j=1$ and $+$ if $j=2$). 
Since if $\tilde{L}-\kappa_j\tilde{E}=0$, then
$\tilde{M}-\kappa_j\tilde{F}=0$, 
we set
\begin{equation}\label{eq:v11}
\bm{V}_j=(-v(\tilde{M}-\kappa_j\tilde{F}),\tilde{L}-\kappa_j\tilde{E}).
\end{equation}
Then $\bm{V}_j$ is a solution of \eqref{eq:p-dir2}, namely
$\bm{V}_j$ is a principal vector with respect to $\kappa_j$.
\begin{lemma}\label{lem:p-vect}
Under the above setting, $df(\bm{V}_1)$ is perpendicular to $df(\bm{V}_2)$ on the set of regular points $U\setminus\{v=0\}$.
\end{lemma}
\begin{proof}
We have 
\begin{equation}\label{eq:p-dir-diff}
df(\bm{V}_j)
=v\left(-(\tilde{M}-\kappa_j\tilde{F})f_u+(\tilde{L}-\kappa_j\tilde{E})h\right)
\end{equation}
for $j=1,2$, where $f_v=vh$. 
We note that $f_u$ and $h$ are linearly independent on $U\setminus\{v=0\}$ and 
$\bm{V}_j$ $(j=1,2)$ are non-zero. 
Thus $-(\tilde{M}-\kappa_j\tilde{F})f_u+(\tilde{L}-\kappa_j\tilde{E})h\neq0$ 
on $U\setminus\{v=0\}$. 
We see that 
\begin{equation*}
\inner{df(\bm{V}_1)}{df(\bm{V}_2)}=v^2\left(\tilde{E}\tilde{M}^2-2\tilde{F}\tilde{L}\tilde{M}+\tilde{G}\tilde{L}^2
+(K\tilde{E}-2H\tilde{L})(\tilde{E}\tilde{G}-\tilde{F}^2)\right),
\end{equation*}
where $K=\kappa_1\kappa_2$ and $2H=\kappa_1+\kappa_2$ are the Gaussian curvature and the mean curvature, respectively. 
By \eqref{eq:K-H1}, it holds that $\inner{df(\bm{V}_1)}{df(\bm{V}_2)}=0$ on $U\setminus\{v=0\}$. 
\end{proof}

\begin{proposition}\label{prop:p-dir}
Let $p$ be a pure-frontal singular point and not 
an umbilic point of a frontal $f$. 
Then both principal vectors $\bm{V}_j$ $(j=1,2)$ 
can be extended as $C^\infty$ vector fields.
Moreover, we have the following.
\begin{enumerate}
\item Suppose that $\kappa_t(p)\neq0$. 
Then $\bm{V}_j\ne0$ $(j=1,2)$ 
and both $\bm{V}_1(p)$ and $\bm{V}_2(p)$ are parallel 
to the null vector $\eta$ at $p$.
\item Suppose that $\kappa_t=0$ along the singular curve through $p$. 
Then there exisits linearly independent vectors 
$\bm{W}_1,\bm{W}_2$ such that
$\bm{V}_j$ is parallel to $\bm{W}_j$ $(j=1,2)$.
\item Suppose that $r_b/3-\kappa_\nu\ne0$ on $\gamma$. 
Then $\gamma$ is a curvature line if and only if $\kappa_t=0$ on $\gamma$.
\end{enumerate}
\end{proposition}
Here, a curve $\gamma$ is a {\it curvature line} if 
$\gamma'$ is a principal vector.
\begin{proof}
Since $p$ is not an umbilic point, $\kappa_j$ $(j=1,2)$ is
a $C^\infty$ function, so do ${\bm{V}}_j$. 
By \eqref{eq:lkemkf}, if $\kappa_t\ne0$, then
both ${\bm{V}}_j\ne0$ $(j=1,2)$,
and they are parallel to the null vector $\eta$ at $p$
by \eqref{eq:v11}.
Therefore we get the first assertion. 

We show (2).
Since $p$ is not an umbilic point, 
one of 
$\tilde L-\kappa_j\tilde E$ $(j=1,2)$ is not zero.
We assume $\tilde L-\kappa_1\tilde E\ne0$ at $p$,
namely $r_b/3-\kappa_\nu>0$.
Then ${\bm{V}}_1\ne0$
We take an orthogonal adapted coordinate system $(U;u,v)$ around $p$. 
Then by the assumption, there exist functions $\tilde{F}_1$ and 
$\tilde{M}_1$ on $U$ such that 
$\tilde{F}=v\tilde{F}_1$ and $\tilde{M}=v\tilde{M}_1$ hold. 
We set
$
\tilde{\bm{W}}_2=
(-(\tilde N_1-\kappa_2\tilde G),\tilde M_1-\kappa_2\tilde F_1).
$
Then $\bm{W}_2=
(-v(\tilde N_1-\kappa_2\tilde G),\tilde M-\kappa_2\tilde F)$
satisfies 
$\bm{W}_2=v\tilde{\bm{W}}_2$, and by the assumption
$\tilde L-\kappa_1\tilde E\ne0$,
it holds that $\tilde N_1-\kappa_2\tilde G\ne0$, namely
$\tilde{\bm{W}}_2\ne0$.
Then the pair $\bm{V}_1$ and $\tilde{\bm{W}}_2$ is 
the desired one.

We show (3).
We assume $\kappa_t=0$ on $\gamma$. We take 
$\tilde{\bm{W}}_2$ in the proof of (2).
Since the $u$-axis is a set of pure-frontal singular point,
$\nu_v=0$ along the $u$-axis. Thus $\nu_{uv}=0$, and we see
$\tilde M_v=-\inner{h_v}{\nu_u}$ on the $u$-axis.
By \eqref{eq:fvfvgv} and $\inner{h}{\nu_u}=0$,
we see $\tilde M_v=\tilde F_1\tilde L$ on the $u$-axis.
Thus by \eqref{eq:invariants} and \eqref{eq:lkemkf},
$\tilde{\bm{W}}_2=
(r_b/3-\kappa_\nu,0)$ when $r_b/3-\kappa_\nu>0$.
If $r_b/3-\kappa_\nu<0$, then interchange $j=1,2$ into $j=2,1$.
\end{proof}
We remark that a similar result for a cuspidal edge is known (see \cite{is-take,tera}). 

Under the condition as in Proposition \ref{prop:p-dir} (2), 
one can take a coordinate system 
$(x_1,x_2)$ on a neighborhood of $p$ such that 
$\partial_{x_i}$ $(i=1,2)$ are parallel to $\bm{V}_i$ (or $\bm{W}_i$) 
by the lemma in \cite[page 182]{kob-nom}. 
We may think of such a coordinate system 
as a {\it curvature line coordinate system} (cf. Definition \ref{def:curvline-frame}).

\begin{corollary}\label{cor:frametarget}
There exist $C^\infty$ maps 
$\bm{e}_j\colon U\to\R^3\setminus\{0\}$ $(j=1,2)$ 
such that $\bm{e}_i$ and $df(\bm{V}_i)$ 
are linearly dependent for $i=1,2$, 
$\inner{\bm{e}_j}{\bm{e}_k}=\delta_{jk}$ $(1\leq j,k\leq2)$, 
where $\delta_{jk}$ is the Kronecker delta. 
\end{corollary}
\begin{proof}
By \eqref{eq:p-dir-diff},
setting
$$
\bm{e}_j=
\dfrac{
-(\tilde{M}-\kappa_j\tilde{F})f_u+(\tilde{L}-\kappa_j\tilde{E})h
}
{|-(\tilde{M}-\kappa_j\tilde{F})f_u+(\tilde{L}-\kappa_j\tilde{E})h|},
$$
we get the assertion for the case of 
$\bm{V}_1=(-v(\tilde{M}-\kappa_1\tilde{F}),\tilde{L}-\kappa_1\tilde{E})
\ne0$ and
$\bm{V}_2=(-v(\tilde{M}-\kappa_2\tilde{F}),\tilde{L}-\kappa_2\tilde{E})
\ne0$.
We can show other cases by similar calculations.
\end{proof}
See \cite{bruce-taribin,bruce-tari,davydov,sotogar} for 
approaches by binary differential equations
for lines of curvatures.

\subsection{Umbilic points}\label{sec:umbpure}
We next focus on umbilic points. 
Let $f$ be a frontal on a neighborhood $U$ of 
a pure-frontal singular point $p$.
The function $\Gamma=H^2-K$ behaves as follows near $p\in S(f)$.

\begin{theorem}\label{thm:isolated-umbilic}
If $p\in S(f)$ is both a pure-frontal singular point and an umbilic point,
then $p$ is a critical point of $\Gamma=H^2-K$.
Moreover, 
if $f$ satisfies 
$r_c(p)\neq0$, and either $3\kappa_\nu'(p)\neq r_b'(p)$ or $\kappa_t'(p)\neq0$, 
then $\Gamma$ is a Morse function with index $0$ or $2$ at $p$, in particular, 
$p$ is an isolated umbilic point of $f$. 
\end{theorem}
In the second assumption, since $r_c(p)\neq0$ is satisfied,
$f$ at $p$ is a $5/2$-cuspidal edge 
(\cite[Fact 2.1, Proposition 3.8]{honda-saji}). 
\begin{proof}
We take an orthogonal adapted coordinate system $(U;u,v)$
around $p$ with $\det(f_u,f_{vv},\nu)(u,0)>0$. 
On the $u$-axis, $H$, $K$ and $\Gamma$ are written as 
$$H=\dfrac{\kappa_\nu}{2}+\dfrac{r_b}{6},\quad 
K=\dfrac{\kappa_\nu r_b}{3}-\kappa_t^2,\quad 
\Gamma=\dfrac{1}{2}\left(\kappa_\nu-\dfrac{r_b}{3}\right)^2+\kappa_t^2.$$
Thus we see that 
$H_u={\kappa_\nu'}/{2}+{r_b'}/{6}$ and
$$
K_u
=\dfrac{\kappa_\nu' r_b+\kappa_\nu r_b'}{3}-2\kappa_t\kappa_t',\quad 
\Gamma_u=\dfrac{1}{2}
\left(\kappa_\nu-\dfrac{r_b}{3}\right)
\left(\kappa_\nu'-\dfrac{r_b'}{3}\right)+2\kappa_t\kappa_t'$$
at $p$. 
If $p$ is an umbilical point of $f$, then $3\kappa_\nu(p)=r_b(p)$ and $\kappa_t(p)=0$. 
Thus $\Gamma_u(p)=0$. 
On the other hand, it is known that
$H_v={r_c}/{48}$ and 
$K_v=r_{\Pi}/{24}$, where $r_{\Pi}=\kappa_\nu r_c$
hold along the $u$-axis (see \cite[Lemma 4.3]{honda-saji}). 
Therefore we have 
$$\Gamma_v=\left(\kappa_\nu+\dfrac{r_b}{3}\right)\dfrac{r_c}{48}-\dfrac{r_\Pi}{24}=0$$
at $p$ since $3\kappa_\nu=r_b$ holds at $p$. 
Thus we get the first assertion.
By direct calculations, it follows that $\Gamma_{uu}$ is
$$
\left(\kappa_\nu'-\dfrac{r_b'}{3}\right)^2
+\left(\kappa_\nu-\dfrac{r_b}{3}\right)
\left(\kappa_\nu''-\dfrac{r_b''}{3}\right)+2(\kappa_t')^2+2\kappa_t\kappa_t''\\
=\left(\kappa_\nu'-\dfrac{r_b'}{3}\right)^2+2(\kappa_t')^2
$$
at $p$ since $2\Gamma=(\kappa_\nu-r_b/3)^2+2\kappa_t^2$ along the $u$-axis, 
$3\kappa_\nu=r_b$ and $\kappa_t=0$ at $p$. 
We next consider $H_{uv}$ and $K_{uv}$. 
Since $H_v=r_c/48$ and $K_v=r_{\Pi}/24$ hold along the $u$-axis, 
$H_{uv}=H_{vu}$ and $K_{uv}=K_{vu}$ are
$H_{uv}={r_c'}/{48}$ and $K_{uv}={r_\Pi'}/{24}$ 
at $p$. 
Thus 
\begin{align*}
\Gamma_{uv}&=2H_uH_v+2HH_{uv}-K_{uv}
=\left(\kappa_\nu'+\dfrac{r_b'}{3}\right)\dfrac{r_c}{48}+\left(\kappa_\nu+\dfrac{r_b}{3}\right)\dfrac{r_c'}{48}-\dfrac{r_\Pi'}{24}\\
&=\left(\kappa_\nu'+\dfrac{r_b'}{3}\right)\dfrac{r_c}{48}+\dfrac{\kappa_\nu r_c'}{24}-\dfrac{\kappa_\nu' r_c+\kappa_\nu r_c'}{24}
=\dfrac{r_c}{48}\left(\dfrac{r_b'}{3}-\kappa_\nu'\right)
\end{align*}
holds at $p$ because the relation $3\kappa_\nu=r_b$ holds at $p$. 
To see $\Gamma_{vv}$ at $p$, we give some calculations in advance.
Since $f_v(u,0)=0$, we have $\tilde{E}_v=2\inner{f_u}{f_{uv}}=0$ along the $u$-axis. 
On the other hand, we see that $\tilde{L}_v=-\inner{f_u}{\nu_u}=0$ at $p$ because $f_{uv}=\nu_{uv}=0$ along the $u$-axis. 
Further, $\tilde{M}=-\inner{h}{\nu_u}=0$ and $\tilde{M}_v=\tilde{F}_v\tilde{L}$ at $p$ since $\tilde{M}(p)=\kappa_t(p)=0$ by Lemma \ref{prop:invariants} 
and $h_v=\tilde{F}_vf_u+(\tilde{G}_v/2)h$ at $p$. 
Using these relations and $\tilde{L}=\tilde{N}_1$, we see that 
\begin{equation*}
2HH_{vv}-K_{vv}=2\tilde{F}_v^2\tilde{N}_1^2+2\tilde{M}_v^2-4\tilde{F}_v\tilde{M}_v\tilde{N}_1
=4\tilde{F}_v^2\tilde{N}_1^2-4\tilde{F}_v^2\tilde{N}_1^2=0
\end{equation*}
holds at $p$.
Thus 
$\Gamma_{vv}(p)=2H_v(p)^2=(r_c(p)/{24})^2/2$ holds. 
Hence the 
determinant of the Hesse matrix $\hess(\Gamma)(p)$ of $\Gamma$ at $p$ is
$$
\det\hess(\Gamma)(p)=\left(\dfrac{r_c(p)}{24}\right)^2
\left(\dfrac{1}{4}\left(\kappa_\nu'(p)
-\dfrac{r_b'(p)}{3}\right)^2+\kappa_t'(p)^2\right)>0.
$$
Hence we have the assertion.
\end{proof}

\section{Principal curvatures near $k$-non-front singular points}
Let $f$ be a frontal and $p$ a $k$-non-front 
singular point of $f$.
\subsection{Behavior of principal curvatures near $k$-non-front singular points}
The following holds for principal curvatures 
defined in \eqref{eq:25-principal}
of a frontal near a $k$-non-front singular point $p$. 
A function $x$ is said to have {\it finite multiplicity at\/} $0$
if there exist an integer $l\geq2$ such that
$x'(0)=\cdots=x^{(l-1)}=0$ and $x^{(l)}\ne0$ at $0$.
\begin{theorem}\label{thm:principal-k-frontal}
Let $f:U\to\R^3$ be a frontal and $p\in U$ a singular point of the first kind of $f$,
where $U$ is sufficient small.
Let $\gamma$ be a singular curve through $p=\gamma(0)$. 
\begin{enumerate}
\item\label{assertion1-conti} 
If $p$ is a $2k$-non-front $(k\geq1)$ singular point of $f$, 
then one of principal curvatures can be extended as a continuous function near $p$ 
and another is unbounded near $p$.
\item\label{assertion2-unbdd} If $p$ is a $(2k+1)$-non-front $(k\geq0)$ 
singular point of $f$, and 
$\kappa_\nu$ does not vanish at $0$, or has a finite multiplicity at $0$,
$\kappa_1$ is coutinuous on $U\setminus \gamma(\{t\leq 0\})$ but unbounded
at $\gamma(\{t\leq 0\})$,
and
$\kappa_2$ is coutinuous on $U\setminus \gamma(\{t\geq 0\})$ but unbounded
at $\gamma(\{t\geq 0\})$,
where $\kappa_i$ $(i=1,2)$ are the principal curvature of $f$.
\end{enumerate}
\end{theorem}
\begin{proof}
Let us take an orthogonal adapted coordinate system $(U;u,v)$ centered at $p$. 
Then one can rewrite principal curvatures as 
$$\kappa_1=\dfrac{2(\tilde{L}\tilde{N}-v\tilde{M}^2)}{A-B},\quad 
\kappa_2=\dfrac{2(\tilde{L}\tilde{N}-v\tilde{M}^2)}{A+B},$$
where 
\begin{equation}\label{eq:ab}
A=\tilde{E}\tilde{N}-2v\tilde{F}\tilde{M}+v\tilde{G}\tilde{L},\quad
B=
\sqrt{A^2-4v(\tilde{E}\tilde{G}-\tilde{F}^2)(\tilde{L}\tilde{N}-v\tilde{M}^2)}
\end{equation}
(cf. \cite{tera0,tera}). 
By Lemmas \ref{prop:invariants} and \ref{lem:weingarten}, the function $\psi$ as in \eqref{eq:delta} can be written as  
$$\psi(u)=\tilde{N}(u,0)=\dfrac{\kappa_c(u)}{2}.$$

We first show the assertion \ref{assertion1-conti}. 
By the definition of a $k$-non-front singular point and the division lemma, if $p$ is a $2k$-non-front singular point, 
then there exists a function $\tilde{\kappa}_c(u)$ such that $\kappa_c(u)=u^{2k}\tilde{\kappa}_c(u)$ and $\tilde{\kappa}_c(0)\neq0$. 
Moreover, by the division lemma again, there exists a $C^\infty$ function $\tilde{N}_1$ on $U$ such that 
$\tilde{N}(u,v)=\tilde{N}(u,0)+v\tilde{N}_1(u,v)$. 
By the above arguments, it holds that 
$$\tilde{N}(u,v)=\dfrac{u^{2k}\tilde{\kappa}_c(u)}{2}+v\tilde{N}_1(u,v).$$
Thus the functions $A$ and $B$ given in above are written as 
$$A=\dfrac{u^{2k}\tilde{\kappa}_c(u)}{2}+vX(u,v),\quad B=\sqrt{\dfrac{u^{4k}\tilde{\kappa}_c(u)^2}{4}+vY(u,v)},$$
where $X$ and $Y$ are some functions. 
Therefore we see that 
$$A(u,0)\pm B(u,0)=\dfrac{u^{2k}}{2}(\tilde{\kappa}_c(u)\pm|\tilde{\kappa}_c(u)|)$$
holds for $u\neq0$. 
On the other hand, by Lemma \ref{prop:invariants} and the division lemma, 
$\tilde{L}(u,v)=\kappa_\nu(u)+v\tilde{L}_1(u,v)$ for some $C^\infty$ function $\tilde{L}_1$. 
Thus we have 
$$\tilde{L}\tilde{N}-v\tilde{M}^2=\dfrac{u^{2k}\kappa_\nu(u)\tilde{\kappa}_c(u)}{2}+vZ(u,v),$$
where $Z$ is some function. 
Since $f$ at $(u,0)$ $(u\ne0)$ is a front,
if $\kappa_i$ $(i=1$ or $2)$ is well-defined,
then $\kappa_i$ is continuous around $(u,0)$.
Hence it is sufficient to show the well-definedness 
of ${2(\tilde{L}\tilde{N}-v\tilde{M}^2)}/{(A\pm B)}$ at $(u,0)$.
We have 
$$\dfrac{2(\tilde{L}\tilde{N}-v\tilde{M}^2)}{A\pm B}(u,0)
=\dfrac{\kappa_\nu(u)
\tilde{\kappa}_c(u)}{(\tilde{\kappa}_c(u)\pm|\tilde{\kappa}_c(u)|)}$$
for $u\neq0$. 
Since $\tilde{\kappa}_c(0)\neq0$ and continuity of $\tilde{\kappa}_c$,
it holds that
$\tilde{\kappa}_c(u)>0$ or $\tilde{\kappa}_c(u)<0$. 
Thus one of principal curvatures is continuous along the $u$-axis, 
and hence we have the first assertion. 

We next show the assertion \ref{assertion2-unbdd}. 
If $p$ is a $(2k+1)$-non-front singular point, then we see that 
there exists a function $\tilde{\kappa}_c(u)$ such that 
$\kappa_c(u)=u^{2k+1}\tilde{\kappa}_c(u)$ and $\tilde{\kappa}_c(0)\neq0$. 
By the assumption, $\kappa_\nu(u)=u^ly(u)$
for some non-zero function $y$.
By the similar discussion, we have the conclusion.
\end{proof}
By this theorem, both principal curvatures are unbounded near a 
cuspidal cross cap $(k=1)$ if $\kappa_\nu\neq0$. 

\begin{remark}
In \cite{msuy}, notions of a {\it rational boundedness} and a {\it rational continuity} for (unbounded) functions are defined. 
Using these contexts, the unbounded principal curvature of a frontal with a $2k$-non-front singular point 
is always rationally bounded at the singular point.   
Moreover, both principal curvatures of a frontal are rationally bounded at a $(2k+1)$-non-front singular point. 
\end{remark}

\begin{example}
Let $f_1=(u,v^2,u^2+uv^3)$ and $f_2=(u,v^2,u^2+(u^2-v^2)v^2)$.
Then $(0,0)$ is a cuspidal cross cap singularity 
(a $1$-non-front singularity) of $f_1$, and
$(0,0)$ is a cuspidal $S_1^-$ singularity 
(a $2$-non-front singularity) of $f_2$.
Then both the limiting normal curvatures of $f_1$ and $f_2$ are 
$2(\ne0)$ at $(0,0)$.
The functions $A\pm B$ of $f_i$ $(i=1,2)$ as in \eqref{eq:ab} on the $u$-axis
are
$$
3(u^i\pm |u^i|)\sqrt{1+4u^2}.
$$
Thus Theorem \ref{thm:principal-k-frontal} in these functions are
verified.
Graphs of the principal curvatures of $f_1$ and $f_2$ are drawn
in Figures \ref{fig:prif1} and \ref{fig:prif2}.
\begin{figure}[htbp]
\begin{center}
\includegraphics[width=.3\linewidth]{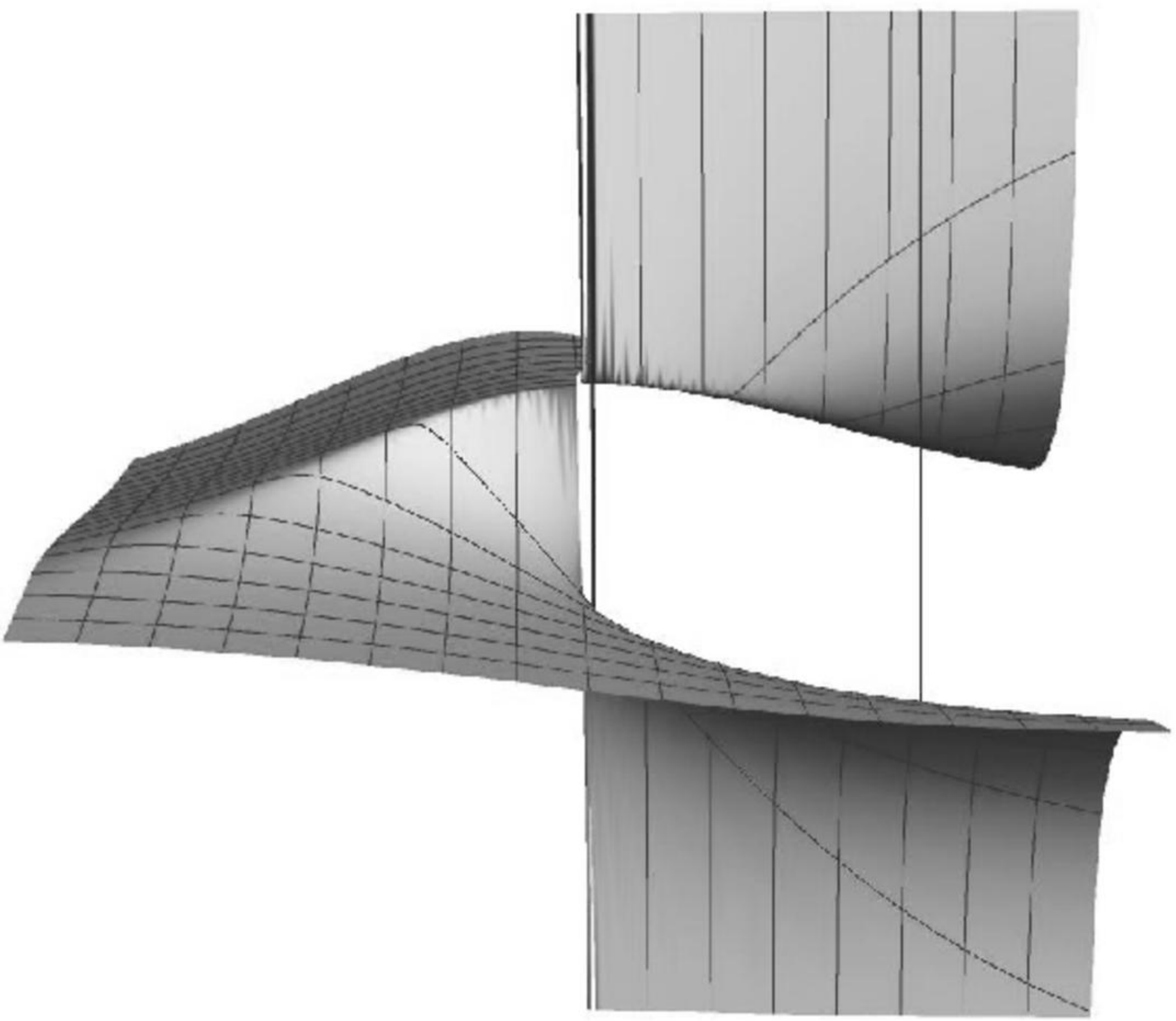}
\includegraphics[width=.3\linewidth]{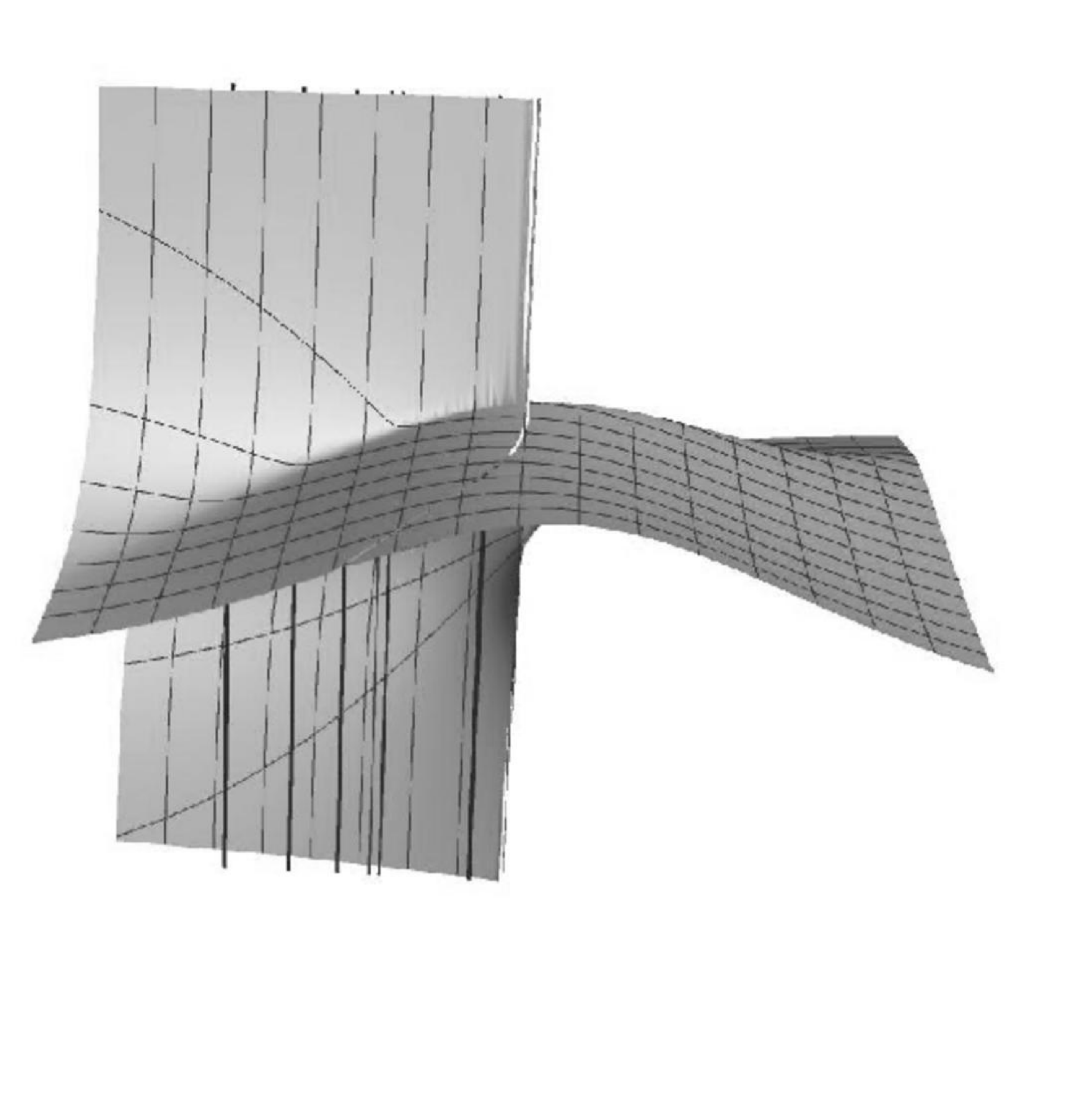}
\end{center}
\caption{Principal curvatures of $f_1$.}
\label{fig:prif1}
\end{figure}
\begin{figure}[htbp]
\begin{center}
\includegraphics[width=.3\linewidth]{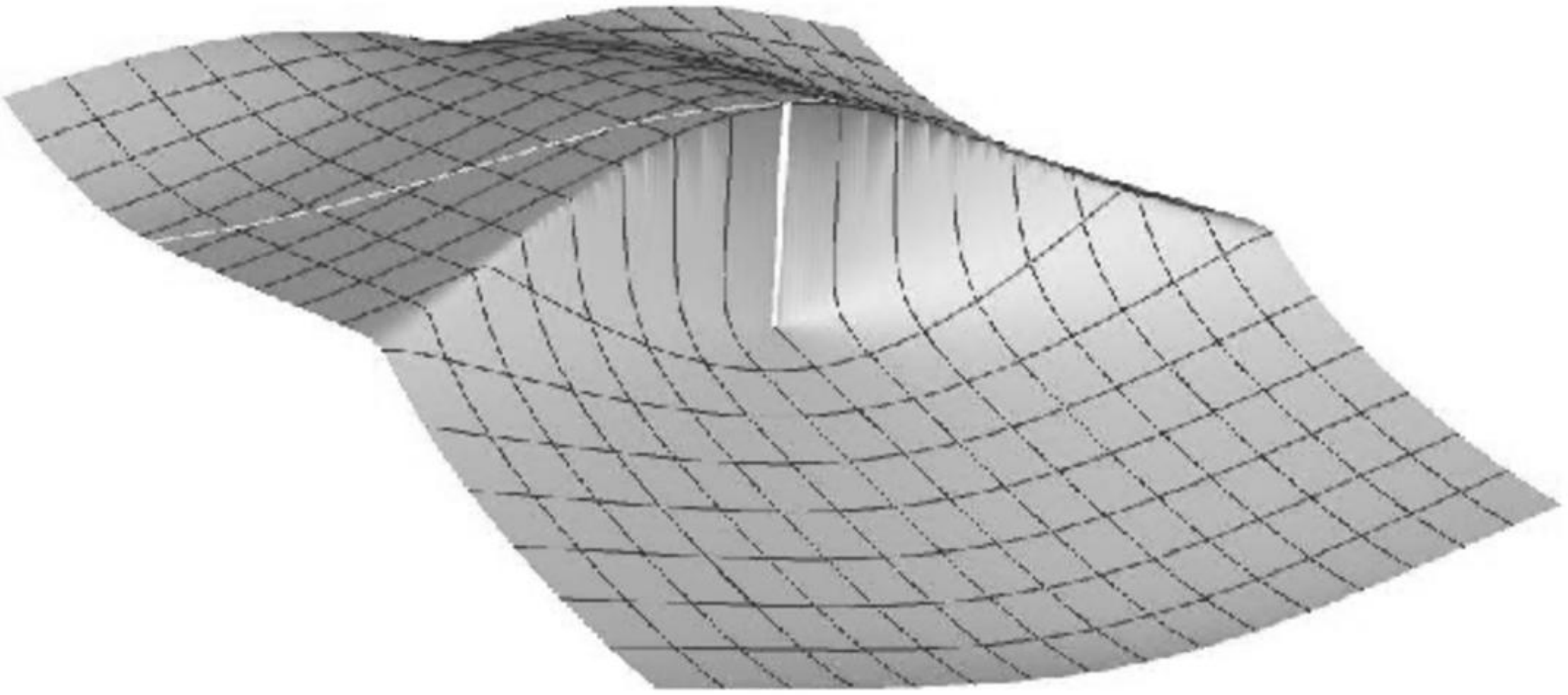}
\includegraphics[width=.3\linewidth]{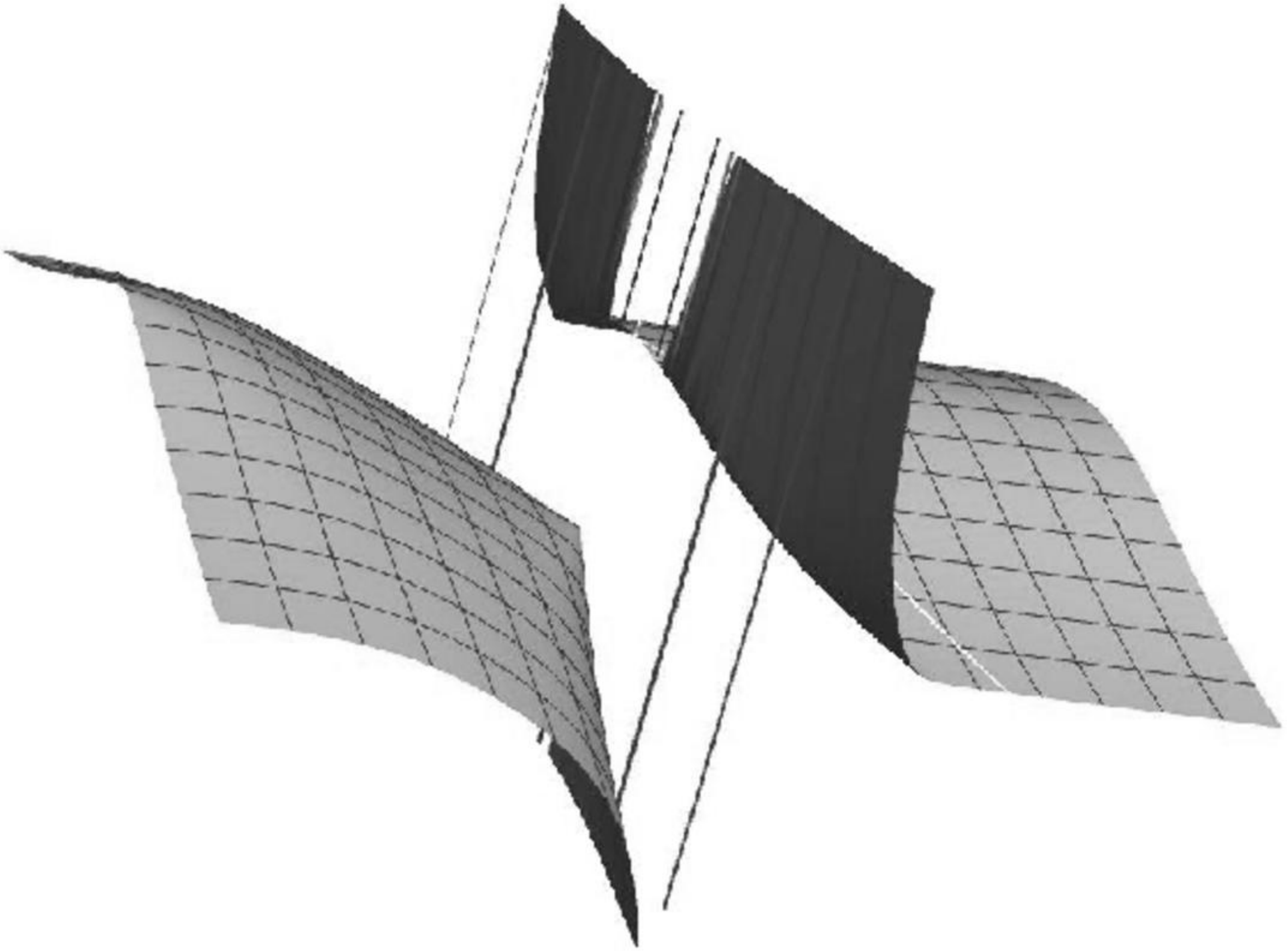}
\end{center}
\caption{Principal curvatures of $f_2$.}
\label{fig:prif2}
\end{figure}
\end{example}

We next consider the case that the Gaussian curvature $K$ of a front is bounded near a $k$-non-front singular point. 
\begin{proposition}\label{prop:bdd-gauss1}
Let $f$ be a frontal and $p$ a $k$-non-front singular point. 
If the Gaussian curvature $K$ of $f$ is bounded near $p$, 
then $K$ is non-positive at $p$. 
\end{proposition} 
\begin{proof}
Let us take an orthogonal adapted coordinate system $(U;u,v)$ around $p$. 
Suppose that the Gaussian curvature $K$ is bounded on $U$. 
Then it is known that the limiting normal curvature $\kappa_\nu$ vanishes along the $u$-axis. 
Thus there exists a $C^\infty$ function $\tilde{L}_1$ on $U$ such that $\tilde{L}=v\tilde{L}_1$ 
by Lemma \ref{prop:invariants} and the division lemma. 
In this case, $K$ can be written as 
$K=(\tilde{L}_1\tilde{N}-\tilde{M}^2)/(\tilde{E}\tilde{G}-\tilde{F}^2)$. 
Since $\tilde{N}(p)=\kappa_c(p)/2=0$, $K(p)=-\tilde{M}(p)^2(=-\kappa_t(p)^2)\leq0$ holds. 
\end{proof}

By the proof of this proposition, the following assertion holds. 
\begin{corollary}\label{cor:bdd-gauss1}
Suppose that the Gaussian curvature $K$ of a frontal is bounded near 
a $k$-non-front singular point $p$.
Then $K(p)=0$ if and only if $\kappa_t(p)=0$.
\end{corollary}
For the case of a cuspidal edge $p$, if the Gaussian curvature $K$ is bounded near $p$, 
then it holds that $4K(p)=-4\kappa_t(p)^2-\kappa_s(p)\kappa_c(p)^2$ (cf. \cite{msuy}). 
Moreover, this quantity relates to the value of the Gaussian curvature of a focal surfaces with respect to unbounded principal curvature (\cite{tera2}).

\subsection{Umbilicity of a frontal at $k$-non-front singular points}
\label{sec:umbknon}
We consider umbilicity of a frontal. 
Let $f$ be a frontal, and let $p$ be a $k$-non-front singular point. 
Then we take an orthogonal adapted coordinate system $(U;u,v)$ around $p$. 
On this coordinate system, the function $\Gamma=H^2-K$ can be 
written as in \eqref{eq:umbilicity2} on $U\setminus\{v=0\}$, 
where $K$ and $H$ are the Gaussian curvature and the mean curvature, respectively. 
By this expression, a function $\tilde{\Gamma}=4\lambda^2 \Gamma$ can be extended as a $C^\infty$ function on $U$ 
since $\lambda=v\det(f_u,h,\nu)$. 
On the set of regular points $U\setminus\{v=0\}$, $\Gamma=0$ is equivalent to $\tilde{\Gamma}=0$. 
Thus in this case, we say that a point $q\in U$ is an {\it umbilic point} of a frontal 
$f$ with a $k$-non-front singular point if $\tilde{\Gamma}(q)=0$. 
\begin{proposition}\label{prop:isolated-umbilic-k-frontal}
Let $f$ be a frontal with a $k$-non-front singular point $p$. 
Then $p$ is an umbilic point of $f$, and a critical point of $\tilde{\Gamma}$. 
Moreover, if $f$ satisfies $\kappa_t\kappa_c'\neq0$ at $p$, 
then $\tilde{\Gamma}$ is a Morse function with index $0$ or $2$ at $p$, 
in particular, 
$p$ is an isolated umbilic point of $f$. 
\end{proposition}
\begin{proof}
Take an orthogonal adapted coordinate system $(U;u,v)$ centered at $p$. 
Then the first assertions follow immediately by \eqref{eq:umbilicity2} 
and the fact that $\tilde{N}(p)=0$. 
By a direct calculation, we have 
$$\tilde{\Gamma}_{uu}=2\tilde{N}_u^2,\quad 
\tilde{\Gamma}_{uv}=2\tilde{N}_u(\tilde{N}_v-\tilde{L}),\quad 
\tilde{\Gamma}_{vv}=8\tilde{M}^2+2(\tilde{N}_v-\tilde{L})^2$$
at $p$. 
Thus 
$\det\hess(\tilde{\Gamma})(p)
=16\tilde{M}(p)^2\tilde{N}_u(p)^2$.
By Lemma \ref{prop:invariants}, it holds that 
$\tilde{M}=\kappa_t$ and $\tilde{N}_u=\kappa_c'$ at $p$. 
This shows the assertion.
\end{proof}

\section{Curvature line frames}\label{sec:curvlineframe}
Motivating the above discussion, 
we introduce a notion `curvature line frames' on frontals. 
Let $f\colon U\to\R^3$ be a frontal and $\nu$ its Gauss map, 
where $U$ is an open set of $\R^2$. 
For a vector field $\x$, we set $f_{\x}=df(\x)$ stands for
the directional derivative of $f$ by $\x$.

\subsection{Curvature line frames}
\begin{definition}\label{def:curvline-frame}
A {\it curvature line frame generator} of a frontal $f$ is a pair of 
vector fields $\u_1$ and $\u_2\in \mathfrak{X}(U)$ satisfying
\begin{enumerate}
\item\label{cond1} $\{f_{\u_1},f_{\u_2}\}$ gives a 
basis of $\nu^{\perp}$ on $U\setminus S(f)$, and
$\inner{f_{\u_1}}{f_{\u_2}}=0$.
\item\label{cond2} 
the pairs 
$\{f_{\u_i},\nu_{\u_i}\}$ $(i=1,2)$
are linearly dependent.
\end{enumerate}
A {\it curvature line frame corresponding to\/ $\{\u_1,\u_2\}$} of 
a frontal $f$ is a pair of sections 
$\{\e_1,\e_2\}$ of a vector bundle $\nu^{\perp}$ satisfying 
that $\{\u_1,\u_2\}$ is a curvature line frame generator and
\begin{enumerate}
\item\label{cond21} $\{\e_1,\e_2\}$ gives an 
orthonormal frame of $\nu^{\perp}$ for any point in $U$,
\item\label{cond22} 
the pairs 
$\{f_{\u_i},\e_i\}$ $(i=1,2)$
are linearly dependent.
\end{enumerate}
The directions defined by $\u_1,\u_2$ are called the
{\it principal direction generators}.
The directions defined by $\e_1,\e_2$ are called the
{\it principal directions}.
A coordinate system $(x_1,x_2)$ is called a {\it curvature line coordinate}
if for a curvature line frame generator $\{\u_1,\u_2\}$,
$\partial_{x_i}$ and $\u_i$ are linearly dependent
on the set of regular points of $f$ for $i=1,2$.
Each integral curve of curvature line generator
is called the {\it line of curvature}.
\end{definition}
We remark that curvature line frame generator
might be linearly dependent on the set of singular points.
Existence of a curvature line coordinate system around a non-umbilic point
of a regular surface is well known.
In Proposition \ref{prop:p-dir}, we gave 
not only existence of
a curvature line coordinate
but also a explicit construction of it
near a pure-frontal singular point of a frontal.
The existence of curvature line coordinate systems
for front is known \cite{murataumehara}.
Following the argument in \cite{murataumehara},
we give a curvature line frame explicitly
near a non-degenerate singular points of fronts
as follows.
Since $f$ is a front, there exists a constant $t\in \R$ such that a parallel surface $f_t=f+t\nu$ of $f$ is regular at $p$. 
We note that $\nu$ is also the Gauss map for $f_t$. 
Since $p$ is a non-degenerate singular point of $f$, $\rank df_p=1$. 
This implies that $p$ is not an umbilic point of $f_t$. 
Thus a curvature line coordinate system $(u,v)$ for $f_t$ exists on some neighborhood $V(\subset U)$ of $p$. 
Since either $f_u\neq0$ or $f_v\neq0$ holds at $p$, one can assume that $f_u(p)\neq0$. 
The following holds:
\begin{lemma}\label{lem:fv=0}
On $S(f)\cap V$, $f_v=0$ holds.
\end{lemma}
\begin{proof}
Since $(u,v)$ is a curvature line coordinate system for $f_t$ on $V$, 
we have 
$$\inner{(f_t)_u}{(f_t)_v}=\inner{f_u+t\nu_u}{f_v+t\nu_v}=0.$$
Taking a limit, we have $\lim_{t\to0}\inner{(f_t)_u}{(f_t)_v}=\inner{f_u}{f_v}=0$ on $V$. 
Since $f_u\neq0$, the kernel $\ker df$ is spanned by $a\partial_u+\partial_v$ on $S(f)\cap V$, where $a$ is a $C^\infty$ function on $V$. 
This means that 
$$af_u+f_v=0$$
holds on $S(f)\cap V$. 
Calculating inner product $\inner{af_u+f_v}{f_u}$, it follows that 
$$\inner{af_u+f_v}{f_u}=a\inner{f_u}{f_u}+\inner{f_u}{f_v}=a\inner{f_u}{f_u}=0$$
on $S(f)\cap V$ because $\inner{f_u}{f_v}=0$ on $V$. 
Since $f_u\neq0$, we see that $a=0$ on $S(f)\cap V$, 
and hence we have the conclusion.
\end{proof}
Since $f_v=0$ on $S(f)\cap V$ by Lemma \ref{lem:fv=0}, 
there exists a $C^\infty$ map $\omega\colon V\to\R^3$ such that 
\begin{equation}\label{eq:fvlambdaomega}
f_v=\hat{\lambda}\omega,
\end{equation}
where $\hat{\lambda}$ is an identifier of singularities. 
The exterior derivative of $\lambda$ is calculated as 
$$d\lambda=d\det(f_u,f_v,\nu)=d(\hat{\lambda}\det(f_u,\omega,\nu))=d\hat{\lambda}\det(f_u,\omega,\nu)+\hat{\lambda}\cdot d(\det(f_u,\omega,\nu)).$$
Since $p$ is a non-degenerate singular point, $d\lambda(p)\neq0$, in particular $d\hat{\lambda}(p)\neq0$. 
Thus $\det(f_u,\omega,\nu)$ does not vanish at $p$, 
and hence we see $\omega(p)\neq0$. 
Therefore $\det(f_u,\omega,\nu)$ does not vanish near $p$, and hence $\{f_u,\omega,\nu\}$ gives an moving frames along $f$ at least locally. 
We set 
\begin{equation}\label{eq:e1e2}
\e_1=\dfrac{f_u}{|f_u|},\quad \e_2=\dfrac{\omega}{|\omega|}.
\end{equation}
\begin{lemma}
The frame $\{\e_1, \e_2\}$ is a curvature line frame.
\end{lemma}
\begin{proof}
By definition, $f_u$ and $\e_1$ are linearly dependent, and also
$f_v$ and $\e_2$ are linearly dependent on $V$. 
Since the coordinate system $(u,v)$ is 
a curvature line coordinate system of $f_t$,
$\nu_u$ is linearly dependent to $(f_t)_u=f_u+t\nu_u$, 
and hence $\e_1$ and $\nu_u$ are linearly dependent. 
Similarly, $\nu_v$ and $(f_t)_v=f_v+t\nu_v$ are linearly dependent. 
Therefore $\nu_v$ and $\lambda\omega+t\nu_v$ are linearly dependent. 
This implies that $\omega$ is linearly dependent to 
$\nu_v$ when $\hat{\lambda}\neq0$. 
By the continuity, $\omega$ and $\nu_v$ are linearly dependent on $V$.
\end{proof}
We remark that we used the non-degeneracy 
in \eqref{eq:fvlambdaomega} and 
linearly independence of $f_u$ and $\omega$. 

If $p$ is a pure-frontal singular point of a frontal $f$, then 
as in Proposition \ref{prop:p-dir},
there exists
a curvature line frame generator, and
as in Corollary \ref{cor:frametarget},
there exists
a curvature line frame.

\subsection{Frenet equation}
Let $f\colon U\to\R^3$ be a frontal, 
$\nu$ its Gauss map.
We assume that $f$ at $p$ is a front or
$p$ is a singular point of pure-frontal.
We take a curvature line frame generator $\{\u_1,\u_2\}$,
and corresponding curvature line frame $\{\e_1,\e_2\}$.
Then the fundamental equations are 
\begin{equation}\label{eq:frenet}
\begin{pmatrix} \e_1 \\ \e_2 \\ \nu \end{pmatrix}_{\u_1}=
\begin{pmatrix} 0 & x_1 & x_2 \\ -x_1 & 0 & x_3 \\ -x_2 & -x_3 & 0 \end{pmatrix}
\begin{pmatrix} \e_1 \\ \e_2 \\ \nu \end{pmatrix},\quad 
\begin{pmatrix} \e_1 \\ \e_2 \\ \nu \end{pmatrix}_{\u_2}=
\begin{pmatrix} 0 & y_1 & y_2 \\ -y_1 & 0 & y_3 \\ -y_2 & -y_3 & 0 \end{pmatrix}
\begin{pmatrix} \e_1 \\ \e_2 \\ \nu \end{pmatrix}.
\end{equation}
Since $\nu_{\u_1}$ (\/resp. $\nu_{\u_2}$\/) 
is parallel to $\e_1$ (\/resp. $\e_2$\/), we have 
$x_3=0$ (\/resp. $y_2=0$\/). 
We call this equation as in \eqref{eq:frenet} the {\it Frenet equation}. 
We set 
$$X=\begin{pmatrix} 0 & x_1 & x_2 \\ -x_1 & 0 & 0 \\ -x_2 & 0 & 0 \end{pmatrix},\quad 
Y=\begin{pmatrix} 0 & y_1 & 0 \\ -y_1 & 0 & y_3 \\ 0 & -y_3 & 0 \end{pmatrix}.$$
The integrability condition for \eqref{eq:frenet} is 
$$X_{\u_2}+XY=Y_{\u_1}+YX.$$
By this condition, we have 
\begin{align}
\begin{aligned}\label{eq:conditions}
-x_2y_3+(x_1)_{\u_2}-(y_1)_{\u_1}&=0,\\
x_1y_3+(x_2)_{\u_2}&=0,\\
x_2y_1+(y_3)_{\u_1}&=0.
\end{aligned}
\end{align}
The first equation in \eqref{eq:conditions} is the {\it Gauss equation}, 
and the second and the third equations in \eqref{eq:conditions} are the {\it Codazzi equations}.
In \cite{framed}, invariants of surfaces with singularities using
general moving frame are studied.

\section{Ribaucour transformation of frontals}\label{sec:rib}
In this section, as an application of the extension
of line of the curvature, we consider Ribaucour transformations
of frontals.
A Ribaucour transformation is a transformation of regular surfaces
which preserves the line of curvatures.
It has been attracting attention from the view point of
differential geometry of surfaces \cite{ribhyper,daytoj,masongud}.
The classical definition requires the curvature line
coordinate systems on surfaces.
Using our curvature line frame on frontal, 
it is natural to consider Ribaucour transformations instead of
curvature line coordinate systems.

\subsection{Definition of Ribaucour transformation} 
We give a definition of the Ribaucour transformation for frontals by using moving frame. 
\begin{definition}\label{def:Rib}
Let $f\colon U\to\R^3$ and $\tilde{f}\colon\wtil{U}\to\R^3$ be frontals, 
$\nu$ (resp. $\tilde \nu$) the Gauss map of $f$ (resp. $\tilde f$).
Let $\{\u_1,\u_2\}$ (resp. $\{\tilde\u_1,\tilde\u_2\}$) be 
a curvature line frame generator of $f$ (resp. $\tilde f$) on $U$. 
Then $\tilde{f}$ is a {\it Ribaucour transformation} of $f$ 
if there exist a $C^\infty$ function $h\colon U\to\R$ 
and diffeomorphism $\psi\colon U\to\wtil{U}$ such that for any $p\in U$, 
\begin{enumerate}
\item $f(p)+h(p)\nu(p)=\til{f}(\psi(p))+h(p)\til{\nu}(\psi(p))$,
\item $d\psi_p(\u_i)$ and $\tilde\u_i$ are linearly dependent
for $i=1,2$.
\end{enumerate}
We call a map $p\mapsto f(p)+h(p)\nu(p)$ the {\it center map}. 
\end{definition}
A different approach to singularities of Ribaucour transformations
of a regular surface is given in \cite{ogata}.

\subsection{Equations of Ribaucour transformation}
Following arguments in \cite{ribhyper},
we consider equations which give the Ribaucour transformation of a frontal
when the set $\{q\in U\ |\ f(q)+h(q)\nu(q)\ \text{is regular}\}$ is dense.
Let $f\colon U\to\R^3$ and $\til{f}\colon \wtil{U}\to\R^3$ be fronts or frontals
with only pure-frontal singular points. 
Take a curvature line coordinate $(u,v)$ on $U$ and a curvature line frame $\{\e_1,\e_2,\nu\}$ along $f$. 
Then there exists functions $k_1,k_2,l_1,l_2\colon U\to\R$ such that 
\begin{equation}\label{eq:Rib1}
f_u=k_1\e_1,\quad f_v=k_2\e_2,\quad \nu_u=l_1\e_1,\quad \nu_v=l_2\e_2.
\end{equation}
On the other hand,  since $\{\e_1,\e_2,\nu\}$ gives an orthonormal basis of $\R^3$,
\begin{equation}\label{eq:Rib2}
\til{\nu}=b_1\e_1+b_2\e_2+b_3\nu\quad (b_1^2+b_2^2+b_3^2=1)
\end{equation}
holds for some $b_1,b_2,b_3$,
where $\til{\nu}$ is the Gauss map of $\til{f}$. 
If $\til{f}$ is the Ribaucour transformation, then $f(p)+h(p)\nu(p)=\til{f}(\psi(p))+h(p)\til{\nu}(\psi(p))$ for any $p\in U$, 
where $h$ and $\psi$ are as in Definition \ref{def:Rib}. 
Thus  by \eqref{eq:Rib1} and \eqref{eq:Rib2}, we have 
\begin{align*}
\begin{aligned}\label{eq:tilde_u}
\til{f}_u&=\til{f}_x x_u+\til{f}_y y_u
=f_u+h_u\nu+h\nu_u-h_u\til{\nu}-h\til{\nu}_u\\
&=(k_1+hl_1)\e_1+h_u\nu-h_u\til{\nu}-h\til{\nu}_u,
\end{aligned}
\end{align*}
where we set $\psi(u,v)=(x(u,v),y(u,v))$ and $\til{\nu}_u$ means $\til{\nu}_u=\til{\nu}_x x_u+\til{\nu}_y y_u$. 
By a direct calculation, we see that 
\begin{equation}\label{eq:inner1}
\inner{\til{f}_u}{\til{\nu}}=(k_1+hl_1)b_1+h_u(b_3-1)=0
\end{equation}
holds by \eqref{eq:Rib2}. 
Similarly, we have 
\begin{align*}
\begin{aligned}\label{eq:tilde_v}
\til{f}_v&=\til{f}_x x_v+\til{f}_y y_v
=f_v+h_v\nu+h\nu_v-h_v\til{\nu}-h\til{\nu}_v\\
&=(k_2+hl_2)\e_2+h_v\nu-h_v\til{\nu}-h\til{\nu}_v
\end{aligned}
\end{align*}
by \eqref{eq:Rib1}, where $\til{\nu}_v=\til{\nu}_x x_v+\til{\nu}_y y_v$. 
Therefore it follows that 
\begin{equation}\label{eq:inner2}
\inner{\til{f}_v}{\til{\nu}}=(k_2+hl_2)b_2+h_v(b_3-1)=0
\end{equation}
by \eqref{eq:Rib2}. 
If $b_3-1=0$, then we have $\nu=\til{\nu}$ by \eqref{eq:Rib2}. 
This implies that $f=\til{f}$, and hence this contradicts 
that $\til{f}$ is a Ribaucour transformation of $f$. 
Hence
\begin{equation}\label{eq:huhv}
h_u=m_1(k_1+hl_1),\quad h_v=m_2(k_2+hl_2),\quad
m_i=-\dfrac{b_i}{b_3-1}\quad(i=1,2)
\end{equation}
hold on $U$ by \eqref{eq:inner1} and \eqref{eq:inner2}.
We have the following:
\begin{proposition}
Let $f\colon U\to\R^3$ and $\tilde{f}\colon\wtil{U}\to\R^3$ be frontals, 
and let $\tilde f$ is a Ribaucour transformation of $f$.
Under the notation above,
if the set of regular points of the center map $c_f$ is dense,
then 
\begin{equation}\label{eq:Ribaucour}
\tilde\nu_v\cdot\e_1=\dfrac{b_1}{b_3-1} \tilde\nu_v\cdot\nu,\quad 
\tilde\nu_u\cdot\e_2=\dfrac{b_2}{b_3-1} \tilde\nu_u\cdot\nu
\end{equation}
hold on $U$.
\end{proposition}
These equations are called {\it Ribaucour equations}. 
\begin{proof}
We see that 
\begin{align*}
(f+h\nu)_u&=(k_1+hl_1)\e_1+h_u\nu=(k_1+hl_1)(\e_1+m_1\nu),\\
(f+h\nu)_v&=(k_2+hl_2)\e_2+h_v\nu=(k_2+hl_2)(\e_2+m_2\nu)
\end{align*}
by \eqref{eq:huhv}. 
We note that $\e_i+m_i\nu\neq0$ ($i=1,2$) holds since 
$\e_i$ and $\nu$ are linearly independent. 
Thus the set of singular points $S(c_f)$ of the center map 
$c_f\colon U\to\R^3$ of $f$ defined by $c_f(u,v)=f(u,v)+h(u,v)\nu(u,v)$ 
is $S(c_f)=S_1\cup S_2$, where 
$$S_i=\{(u,v)\in U\ |\ k_i(u,v)+h(u,v)l_i(u,v)=0\}\quad (i=1,2).$$
Since $\psi(u,v)=(x(u,v),y(u,v))$ is also a curvature line coordinate 
system of $\til{f}$ on $U$, it holds that 
$x_v(u,v)=y_u(u,v)=0$, in particular,
$\inner{\til{f}_u}{\til{\nu}_v}=\inner{\til{f}_v}{\til{\nu}_u}=0.$
We now set 
\begin{equation}\label{eq:dtilnu}
\til{\nu}_u=L_1^1\e_1+L_1^2\e_2+L^3_1\nu,\quad \til{\nu}_v=L_2^1\e_1+L_2^2\e_2+L^3_2\nu.
\end{equation}
By direct calculations using \eqref{eq:Rib2} and \eqref{eq:huhv},
\begin{align*}
\til{f}_u&=(f+h\nu)_u-h_u\tilde{\nu}-h\tilde{\nu}_u
=(k_1+hl_1)(\e_1+m_1\nu)-m_1(k_1+hl_1)\til{\nu}-h\til{\nu}_u\\
&=(k_1+hl_1)(\e_1+m_1(\nu-\til{\nu}))-h\til{\nu}_u,\\
\til{f}_v&=(f+h\nu)_v-h_v\til{\nu}-h\til{\nu}_v
=(k_2+hl_2)(\e_2+m_2\nu)-m_2(k_2+hl_2)\til{\nu}-h\til{\nu}_v\\
&=(k_2+hl_2)(\e_2+m_2(\nu-\til{\nu}))-h\til{\nu}_v
\end{align*}
holds.
Since $\inner{\til{\nu}}{\til{\nu}_u}
=\inner{\til{\nu}}{\til{\nu}_v}
=\inner{\til{\nu}_u}{\til{\nu}_v}
=\inner{\til{f}_u}{\til{\nu}_v}
=\inner{\til{f}_v}{\til{\nu}_u}=0$, 
we have 
$$
\inner{\til{f}_u}{\til{\nu}_v}=(k_1+hl_1)(L_2^1+m_1L_2^3)=0,\quad 
\inner{\til{f}_v}{\til{\nu}_u}=(k_2+hl_2)(L_1^2+m_2L_1^3)=0.
$$
Since $R(c_f)$ is dense, 
$L_2^1+m_1L_2^3=0$ (resp. $L_1^2+m_2L_1^3=0$)
holds and this shows the assertion.
\end{proof}
\subsection{Example}
Here we give an example of Ribaucour transformation
in our sense.
Let $\gamma:(x(u),y(u))$ $(y>0)$ be a planar curve in $I\subset\R$
satisfying
that there exist functions $l(u)$ (possibly taking zero) 
and $\theta(u)$ such that
$$
\gamma'(u)=l(u)(\cos\theta(u),\sin\theta(u)).
$$
This condition is equivalent to that the curve is a frontal.
We consider the surface of revolution of $\gamma$
with respect to the $z$-axis
by
\begin{equation}\label{eq:revol}
T_z(\gamma)=\big(x(u),y(u)\cos v,y(u)\sin v\big).
\end{equation}
We set $T_z(\gamma)(u,v)=s(u,v)$.
The curve $\gamma$ is called the 
{\it profile curve\/} or the {\it generating curve\/}
of $s$.
It is well known that if $l\ne0$, then 
$(u,v)$ is a curvature line coordinate system of $s$.
One can easily see that
$(\partial_u,\partial_v)$ is a curvature line frame generator.

Let us fix a function $\rho(u)$ satisfying
$$
k(u)=\dfrac{\rho'(u)}{l(u)}
$$
is bounded for any $u$, and consider a function
$F:I\times \R^2\to\R$ by
$$
F(u,x,y)=|(x,y)-\gamma(u)|^2-\rho(u)^2.
$$
We consider the envelope of the family $\{f_u^{-1}(0)\}_{u\in I}\subset\R^2$
of this function, where $f_u(x,y)=F(u,x,y)$.
We set $\e=(\cos\theta(u),\sin\theta(u))$ and $\n=(-\sin\theta(u),\cos\theta(u))$
and $X=\alpha \e+\beta \n$, $\gamma=r_{1}\e+r_{2}\n$.
Then since
\begin{equation}\label{ftsiki}
\dfrac{\partial }{\partial u}F
=2\Big((\gamma-X) \cdot \gamma'-\rho \rho'\Big),
\end{equation}
we have
$\big((\alpha \e+\beta \n)-(r_1\e+r_2\n)\big)\cdot \gamma'-\rho\rho'
=l \alpha-l r_{1}-\rho \rho'$, 
where $\bm{x}\cdot\bm{y}$ means the canonical inner product of $\bm{x},\bm{y}\in\R^2$.
Thus $\partial F/\partial u=0$ if and only if $\alpha=r_1-\rho k$.
Moreover, by 
$$
X-\gamma = (\alpha-r_{1})\bm{e}+(\beta-r_{2})\bm{n}
=-\rho k \bm{e}+(\beta-r_{2})\bm{n},
$$
we have
\begin{equation}
(\beta-r_{2})^2 = \rho^2-\rho^2k^2 \label{bsiki}.
\end{equation}
Thus if $|k|\ne 1$, then
$
\beta = \pm\rho \sqrt{|1-k^2|}+r_{2}
$.
Hence we have
\begin{equation}\label{xk05}
X = X_{\pm}
= \gamma-\rho k \bm{e} \pm\rho \sqrt{|1-k^2|}\bm{n}.
\end{equation}
If $|k(u)|=1$ for any $u$, then 
\begin{equation}\label{eq:xk1}
X = X_{\pm}
= \gamma \mp \rho \bm{e}.
\end{equation}
Summarizing up the above, we have the following proposition.
We set the surfaces of revolutions $T_z(X_\pm)$.
\begin{proposition}
If $|k(u)|\ne1$ or $|k(u)|=1$ for any $u$,
The surface of revolution $T_z(X_+)$ 
is a Ribaucour transformation of the surface of revolution $T_z(X_-)$
given in
\eqref{xk05} when $|k(u)|\ne1$ or \eqref{eq:xk1} when $|k(u)|=1$.
Moreover, $T_z(\gamma)$ is the center map of these Ribaucour transformations.
\end{proposition}

\begin{acknowledgements}
The authors thank Joseph Cho, Mason Pember and Gudrun Szewieczek
for fruitful advices, and Maho Ichikawa for helping calculations about
Ribaucour transformations of surfaces of revolution. 
\end{acknowledgements}


\end{document}